\documentclass{article}
\usepackage[utf8]{inputenc}

\usepackage[a4paper, left=3.2cm,top=3cm,right=3.2cm,bottom=3.5cm]{geometry}

\usepackage{concmath}
\usepackage[T1]{fontenc}
\usepackage[utf8]{inputenc}

\usepackage{authblk}
\usepackage{cite}
\usepackage{tabularx}
\usepackage[onehalfspacing]{setspace}

\usepackage{amsmath,amsthm,amssymb}
\usepackage{dsfont}
\usepackage{mathtools}
\usepackage{subcaption}
\usepackage[color=green!40]{todonotes}
\usepackage{graphicx}
\usepackage{tikz}
\usepackage{bbm}
\usepackage{enumitem}
\usepackage{tikz}
\usepackage{booktabs}
\usepackage{framed}
\usepackage{algorithm}
\usepackage{algpseudocode}

\usepackage[binary-units]{siunitx}
\usepackage{eurosym}

\newcommand*{\N}{\mathds{N}}

\newcommand*{\R}{\mathds{R}}

\newcommand{\reg}{\mathcal R}
\newcommand{\funk}{\mathcal F}

\newcommand{\Ko}{\mathbf{K}}

\newcommand{\Lo}{\mathbf L}
\newcommand{\Ao}{\mathbf A}

\newcommand{\Wo}{\mathbf W}
\newcommand{\U}{\mathbb U}
\newcommand{\V}{\mathbb V}
\newcommand{\X}{\mathbb X}
\newcommand{\Y}{\mathbb Y}

\newcommand{\Lc}{\mathcal{L}}

\newcommand{\distance}{\mathcal{D}}
\newcommand{\similarity}{\mathcal{S}}
\newcommand{\tik}{\mathcal{T}}
\newcommand{\net}{\mathcal{N}}
\newcommand{\act}{\sigma}

\newcommand{\al}{\alpha}

\newcommand{\signal}{x}
\newcommand{\data}{y}

\def\plus{{\boldsymbol{\texttt{+}}}}

\DeclarePairedDelimiter{\abs}{\lvert}{\rvert}
\DeclarePairedDelimiter{\norm}{\lVert}{\rVert}
\DeclarePairedDelimiter{\innerprod}{\langle}{\rangle}

\DeclarePairedDelimiter{\set}{\{}{\}}

\DeclareMathOperator{\ran}{ran}

\DeclareMathOperator*{\argmin}{arg\,min}

\DeclareMathOperator*{\crit}{crit}

\DeclareMathOperator*{\relu}{ReLU}

\DeclareGraphicsExtensions{.eps,.pdf,.png,.jpg}

\newtheorem{theorem}{Theorem}

\newtheorem{lemma}[theorem]{Lemma}

\theoremstyle{definition}
\newtheorem{proposition}[theorem]{Proposition}
\newtheorem{example}[theorem]{Example}
\newtheorem{cond}[theorem]{Condition}
\newtheorem{remark}[theorem]{Remark}
\newtheorem{definition}[theorem]{Definiton}

\usepackage{soul}
\colorlet{lred}{red!40}
\colorlet{lgreen}{green!40}
\colorlet{lblue}{blue!40}
\definecolor{bananamania}{rgb}{0.98, 0.91, 0.71}

\numberwithin{equation}{section}
\numberwithin{table}{section}
\numberwithin{figure}{section}
\numberwithin{theorem}{section}

\title{Convergence analysis  of critical point regularization with non-convex regularizers}

\author{Daniel Obmann}
\author{Markus Haltmeier}

\affil{Department of Mathematics, University of Innsbruck\authorcr
Technikerstrasse 13, 6020 Innsbruck, Austria\authorcr
 {\tt \{daniel.obmann, markus.haltmeier\}@uibk.ac.at}\authorcr\mbox{}}

\begin{document}

\maketitle

\begin{abstract} One of the key assumptions in the stability and convergence analysis of variational regularization  is the ability of finding global minimizers.
 However, such an assumption is often not feasible when the regularizer is a black box or non-convex making the search for global minimizers of the involved Tikhonov functional a challenging task. This is in particular the case for the emerging class of learned regularizers defined by  neural networks. Instead, standard minimization schemes are applied which typically only guarantee that a critical point is found. To address this issue, in this paper we study stability and convergence properties of critical points of Tikhonov functionals with a possible non-convex regularizer.  To this end, we introduce the concept of relative sub-differentiability and study its basic properties. Based on this concept, we develop a convergence analysis assuming relative sub-differentiability of the regularizer. The rationale behind the proposed concept is that critical points of the Tikhonov functional are also relative critical points and that for the latter a convergence theory can be developed. For the case where the noise level tends to zero, we derive a limiting problem representing first-order optimality conditions of a related restricted optimization problem. Besides this, we also give a comparison with classical methods and show that the class of ReLU-networks are appropriate choices for the regularization functional. Finally, we provide numerical simulations that support our theoretical findings and the need for the sort of analysis that we provide in this paper.


\medskip

\end{abstract}

\section{Introduction} \label{sec:introduction}

In various scientific fields and applications, such as medical imaging or remote sensing, it is often not possible to obtain the desired quantity of interest directly. Assuming a linear measurement model, recovering the quantity of interest requires solving an inverse problem of the form
\begin{equation}\label{eq:ip}
\data^\delta   = \Ko \signal + \eta^\delta \,,
\end{equation}
where  $\Ko \colon \X \to \Y$ is a linear operator between Hilbert spaces modeling the forward problem, $\eta^\delta$ is the data perturbation, $\data^\delta \in \Y$ is the noisy data and $\signal \in \X$ is the sought for signal. In many cases these problems are ill-posed, meaning that no continuous right inverse of the operator $\Ko$ exists. To overcome such issues several established approaches for the stable approximation of solutions of inverse problems exist.

\subsection{Regularization with non-convex penalties}

Particularly popular regularization techniques   are  variational methods \cite{EngHanNeu96, scherzer2009variational}. These methods recover regularized solutions $\signal_\al^\delta$ as global minimizers of the  Tikhonov functional
\begin{equation} \label{eq:tik-fun}
	\tik_{\al, \data^\delta}(\signal) \coloneqq    \frac{1}{2} \norm{\Ko \signal - \data^\delta}^2  + \al \reg(\signal) \,.
\end{equation}
Here,  $\reg$ is a regularizer which encodes prior information about the desired solution and  $\frac{1}{2} \norm{\Ko \signal - \data^\delta}^2$ plays the role of a data-discrepancy measure. Classically, regularizers  have  been hand-crafted, including $L^2$-penalties,  sparse regularization techniques or total variation   \cite{EngHanNeu96,grasmair2008sparse,acar1994analysis}.  While such hand-crafted regularizers are often convex and hence global minima  can be computed by classical convex optimization, hand-crafted priors typically lack adaptability to available data.

In more recent years,  there has been a shift  to learned and potentially non-convex priors \cite{antholzer2021discretization, li2020nett, lunz2018adversarial, mukherjee2020learned, obmann2021augmented}. It has been observed that these methods often outperform classical methods. Moreover, a full convergence analysis has been  be provided \cite{obmann2021augmented,li2020nett}. However,  such an analysis assumes  minimizers of the Tikhonov  functional to be known  or at least be given within a certain accuracy. For non-convex regularizers such an assumption is unrealistic and global minimizers are challenging to compute. Instead, when trying to find a regularized solution  one often employs minimization algorithms such as gradient descent or variations thereof which converge to critical points  (such as local minimizers close to the initial guess) rather than to global minimizers of the Tikhonov functional. While one could constrain the learned regularizers to only include convex functionals  \cite{amos2017input,mukherjee2020learned} this might result in suboptimal reconstructions when the underlying signal class is inherently non-convex. For such classes  non-convexity of the regularizer can be a highly desirable property and as such a convergence analysis for this case is needed. Importantly, such an analysis should not rely on the strict assumption that the regularized solutions  are global minimizers of the underlying Tikhonov functional.

We briefly mention here that there exist other interesting cases where the Tikhonov functional is non-convex such as for example in the case of a nonlinear forward operator. However, in this paper we only consider the linear case.
Besides this we also mention that there are potentially different ways to deal with non-convexity of Tikhonov functionals, for example, by use of convexification  \cite{klibanov2021convexification}. However, for the learned regularizers we have in mind (see results in Section~\ref{sec:discussion}), the modification of the involved functionals  is nontrivial in general. Besides, the modification of the learned functionals can change the original properties of the learned functional in an unfavorable way.

\subsection{Proposed critical point regularization}

In this paper we present a convergence analysis of critical points  of the Tikhonov  functional $\tik_{\al, \data^\delta}$ for the stable solution of inverse problems  of the form \eqref{eq:ip}.  We refer to any such method which recovers a critical point as regularized solutions as critical point regularization. In fact, we show stability and convergence for  a relaxed notion of critical points. More precisely we study stability and convergence of $\phi$-critical points, namely elements satisfying $0 \in \partial_\phi \tik_{\al, \data^\delta} (\signal_\al^\delta)$. Here $\partial_\phi$ is the $\phi$-relative sub-differential, a novel concept that we introduce and study  in this paper.  Whenever the classical norm-discrepancy is used to measure similarity, as the noise level tends to zero, we show that regularized elements converge to elements  $ \signal^\plus \in \X$ with
\begin{equation}
	-\partial_\phi \reg(\signal^\plus) \cap  \ker(\Ko)^\perp \neq \emptyset \,,
\end{equation}
resembling first order conditions of the constraint optimization problem $\argmin\{ \reg(\signal) \mid  \Ko \signal = \data \}$ defining $\reg$-minimizing solutions.

We give our analysis for more general data discrepancy measures  $\similarity(\signal, \data^\delta)$ for which  $ \norm{\Ko \signal - \data^\delta}^2/2$ is only a special case.
Further, we mention that in \cite{durand2006stability} an analysis of stability for the case of local minima has been done. Opposed to our work the authors of \cite{durand2006stability} restrict themselves to the finite dimensional setting and do not  provide convergence results for the case that the noise-level tends to zero.   Allowing that the underlying space is a general Hilbert space without any restrictions on the dimension has the advantage that the analysis is independent of the dimension and as such applies to any discretization used for practical applications. The precise analysis of the discretization is beyond the scope of this paper and we refer to the corresponding work in conventional variational regularization \cite{poschl2010discretization,antholzer2021discretization}.

Note that whenever the Tikhonov functional is relatively subdifferentiable, then critical points of the Tikhonov functional are also relatively critical points if $\phi$ is constructed accordingly, and the proposed concept yields a convergent regularization for critical points. We will show that this is actually the case, for example, for a class of learned regularizers defined by neural networks.   We are not aware of any other study which includes stability and convergence of critical points and to the best of our knowledge the present analysis is the first to attempt this. 
 
\subsection{Main contributions}

In this paper we introduce the concept of relative sub-differentiability as a generalization of sub-differentiability of convex functions  to the non-convex case. We develop theory for relative sub-differentiability  and show that corresponding  $\phi$-critical points can be found by employing a generalized gradient descent method.   From the  viewpoint of regularization theory we give existence, stability and convergence results for $\phi$-critical points and derive the limiting problem for critical point regularization. As opposed to the convex case where the solutions one obtains are $\reg$-minimizing solutions we get as a limiting problem a related first order optimality condition. As a special case of our analysis we derive stability and convergence results for critical points of differentiable Tikhonov functionals. For example, in this case, we get that $-\reg'(\signal_\plus)$ is in the normal cone of the set of all solutions.

Finally, we provide numerical simulations which support our theoretical findings, in particular the stability, convergence and the limiting problem. Moreover, the results of our numerical simulations show that even in simple cases of non-convex regularizers the assumption of obtaining global minima or even local minima is infeasible thus further emphasizing the need for the analysis we provide in this paper. Besides, the numerical results show that the solutions we obtain cannot be expected to be $\reg$-minimizing solutions and may even be local maxima of the regularizer whenever the initialization is chosen inappropriately and the algorithm does not guarantee that local minima are obtained.

\subsection{Overview}

The rest of the paper is organized as follows. In Section~\ref{sec:rsd} we motivate and introduce the concept of relative sub-differentiability and corresponding $\phi$-critical points.  Moreover, we study basic properties of relative sub-differentiability and show that $\phi$-critical points  can be achieved by employing a generalized gradient descent method. Section~\ref{sec:theory} builds on this concept of relative sub-differentiability and gives a convergence analysis for critical point regularization. Moreover, we take a closer look at the differentiable case and identify the limiting problem in this case. In Section~\ref{sec:numerics} we provide numerical experiments which support our theoretical findings such as stability and convergence. Finally, we conclude the paper by giving a brief summary and outlook in Section~\ref{sec:conclusion}.

\section{Relative sub-differentiability}
\label{sec:rsd}

In this section and in the rest of the paper, unless stated otherwise, we assume that $\X$ is a Banach space, denote by $\X^*$ its dual and by $\innerprod{\cdot, \cdot}$ the dual pairing of $\X$ and $\X^*$, i.e. for $\varphi \in \X^*$ and $\signal \in \X$ we have $\innerprod{\varphi, \signal} = \varphi(\signal)$. Moreover, we denote by $\reg'$ the derivative of any differentiable function $\reg \colon \X \to \R$ and for any similarity measure $\similarity \colon \X \times \Y \to \R$ we denote by $\similarity'$ the derivative with respect to its first argument.

Before giving the crucial definition of  relative sub-differentiability we recall the importance  of classical sub-differentiabilty in the context  of convex functions. Recall that  $r \in \X^*$ is called subgradient of some functional $\funk \colon \X \to \R$ at $\signal \in \X$ if $\funk(\signal) + \innerprod{r, u - \signal} \leq \funk(u) $ for all $ u \in\X $ and that  $\funk$  is sub-differentiable whenever the set of subgradients is non-empty for all $\signal \in \X$. Minimizers $\signal$ of $\funk$ are  characterized by the  optimality condition $0 \in \partial_0 \funk(\signal)$ where $\partial_0 \funk(\signal)$ denotes the set of all subgradients at point $x$. However, sub-differentiability implies convexity. We will therefore develop a relaxed  concept of sub-differentiability relative to some functional $\phi \colon \X \to [0, \infty)$ by replacing the right hand side in the definition of subgradients by $\funk(u)  + \phi(u)$.

\subsection{Definition and basic propertties}

The following concept generalizing sub-differentiablity is also applicable to non-convex  functions.

\begin{definition}[Relative sub-differentiability] \label{def:errorconvex}
Let $\funk \colon \X \to \R$  and $\phi \colon \X \to [0, \infty)$.
\begin{enumerate}[topsep=0em, label=(\alph*)]
\item $r \in \X^*$ is called $\phi$-relative subgradient of $\funk$ at  $\signal \in \X$ if
\begin{equation} \label{eq:errorconvex}
  u \in \X \colon \quad  \funk(\signal) + \innerprod{r, u - \signal} \leq \funk(u) + \phi(u).
\end{equation}
\item

The set  set of all $\phi$-relative subgradients at $\signal$ is denoted by $\partial_\phi \funk(\signal)$ and called $\phi$-relative sub-differential of $\funk$ at $x$.

\item
The functional $\funk$  is called $\phi$-relative sub-differentiable if   $\partial_\phi \funk(\signal)  \neq \emptyset$ for all $\signal \in \X$.

\end{enumerate}
\end{definition}

Some remarks about Definition~\ref{def:errorconvex} are in order.

\begin{remark} \mbox{}
\begin{itemize}[topsep=0em]


 \item We call any such function $\phi$ a bound. It is clear that such a bound cannot be unique, since whenever $\funk$ is a relatively sub-differentiable function with bound $\phi$ then it is also relatively sub-differentiable with bound $\phi + c$ for any $c \in [0, \infty)$.

   \item Choosing $\phi = 0$ we see that any convex and sub-differentiable function $\funk$ is relatively sub-differentiable, i.e. the class of all relative sub-differentiable functions  includes the set of convex sub-differentiable functions.

 \item The relative subgradients depend on the function $\phi$. This shows that whenever we choose a larger $\phi$ then we generally also increase the set of possible relative subgradients, i.e. if $\phi_1 \leq \phi_2$ then $\partial_{\phi_1} \funk \subseteq \partial_{\phi_2} \funk$.

 \item Similar to the concept of subgradients for convex functions, the concept of relative subgradients is a global property since the defining inequality has to hold for any point $u \in \X$.
\end{itemize}
\end{remark}

Another approach of generalizing convexity and subgradients (and as a consequence critical points) is given in \cite{grasmair2010generalized, singer1997abstract} where convexity with respect to a set of functions $W$ is defined. In such a setting $w \in W$ is a subgradient of $\funk$ at $\signal$ whenever $\funk(u) \geq \funk(\signal) + w(u) - w(\signal)$ for any $u \in \X$. As a consequence any critical point, i.e. a point where $0$ is a subgradient, will also be a global minimizer and hence such a generalization cannot be used for our purposes. In \cite{clarke1975generalized} another concept of generalized gradients is discussed. In this setting the definition of the gradient depends only on neighborhoods around the point of interest. As a consequence we cannot expect the critical points to have any global properties which are necessary for the analysis in Section~\ref{sec:theory} hence making this generalization unfit for our analysis. However, it should be noted, that whenever convenient one might substitute any differentiability assumption on the involved functionals with Clarke's generalized gradient concept in any of the following discussions.

In what follows we will assume that $\funk \colon \X \to \R$ is $\phi$-relatively sub-differentiable for some fixed $\phi$. Based on this definition we generalize the concept of critical points as follows.

\begin{definition}[$\phi$-critical points] \label{def:criticalpoint}
We call $\signal \in \X$ a $\phi$-critical point of $\funk$ if  $0 \in \partial_\phi \funk(\signal)$. Moreover, we denote by $\crit_\phi \funk$ the set of all $\phi$-critical  points of $\funk$.
\end{definition}

It should be noted that the definition of $\phi$-critical  points depends on $\phi$ and in practical applications one might not have access to $\phi$. In such cases evaluating or finding relative subgradients might be infeasible. Nevertheless, the concept of $\phi$-critical points is general enough to include an important class of points as the following remark illustrates.

\begin{remark}[Critical points of differentiable functions] \label{re:differentiable}
Let us assume that $\funk \colon \X \to \R$ is a differentiable function which satisfies the inequality $\funk(\signal) + \innerprod{\funk'(\signal), u - \signal} \leq \funk(u) + \phi(u)$ for any $\signal, u \in \X$ and some $\phi \colon \X \to [0, \infty)$.
Then we have $\funk'(\signal) \in \partial_\phi \funk (\signal)$. This shows that in this special case we have access to at least one element of $\partial_\phi \funk$. In particular, any critical point of $\funk$, i.e. a point $\signal \in \X$ with $\funk'(\signal) = 0$, will always yield a $\phi$-critical point of $\funk$ in the sense of Definition~\ref{def:criticalpoint} and hence Definition~\ref{def:criticalpoint} is a generalization of the classical concept of critical points for differentiable functions satisfying above inequality.\\
This shows that for a class of functions we have access to at least one element of the relative subgradient of $\funk$. More importantly, for this class of functions we can make assertions about the points $\signal \in \X$ where $\funk'(\signal) = 0$ holds, i.e. points which are reachable by use of a (minimization) algorithm which guarantees to find a critical point.
\end{remark}

Before we move on, we briefly give a prototypical example of a non-convex function for which a bound $\phi$ can be chosen, such that $\funk'(\signal) \in \partial_\phi \funk(\signal)$.

\begin{remark}[Examples of relative sub-differentiability] \label{re:examples}
We start by giving a simple example of a function which is non-convex, but relatively sub-differentiable. To this end, let $a, b \in \R$ be given and define $\funk(t) = (t + a)^2 (t + b)^2$. It is readily seen, that $\funk(t) + \funk'(t)(s - t)$ is a polynomial of degree $4$ with negative leading coefficient. Hence, this function is bounded from above and the relative sub-differentiability immediately follows by for example choosing $\phi(s) = \sup_t \funk(t) + \funk'(t)(s - t)$. Clearly then, the function $g(t) = \funk(t) + c t^2$ is also sub-differentiable for $c > 0$. The function $g$ is plotted in Figure~\ref{fig:example} on the left side for different parameters $a, b, c$ on a semi-logarithmic scale to emphasize the non-convexity.

\begin{figure}
    \centering
    \includegraphics[scale=0.4]{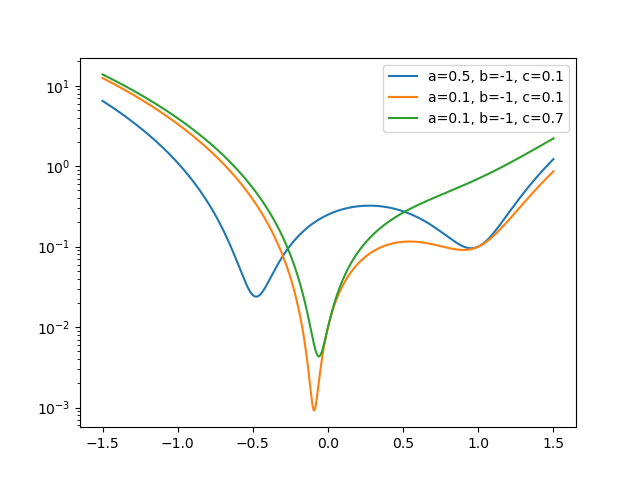}
    \includegraphics[scale=0.4]{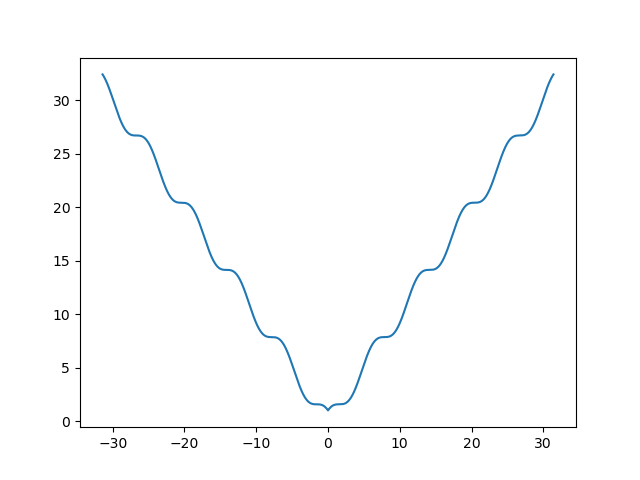}
    \caption{\textbf{Left:} Example of relatively sub-differentiable functions where the classical gradient is contained in the relative sub-gradient. \textbf{Right:} A function for which classical derivative cannot be in the relative sub-gradient.}
    \label{fig:example}
\end{figure}

Now let us consider the function $\funk(t) = \cos(t) + \abs{t}$, see Figure~\ref{fig:example} on the right. Then, due to the coercivity and the existence of critical points ``at infinity'', the derivative of $\funk$ cannot be in the relative sub-gradient of $\funk$ for any $\phi$. This example illustrates what types of functions are not included in the concept of relatively sub-differentiable functions for which the derivative is supposed to lie in the relative subgradient. In particular, the concept of relative sub-differentiability excludes coercive functionals which have critical points ``at infinity''.

\end{remark}

Before discussing how one might obtain $\phi$-critical points of relatively sub-differentiable functions, we list some useful properties of which we make constant use during the rest of the paper.

\begin{lemma}[Basic properties of relative subgradients] \label{lemma:basic}
Let $\funk, \funk_i \colon \X \to \R$ and $\phi, \phi_i \colon \X \to [0, \infty)$ be bounds of $\funk, \funk_i$ for $i = 1, \dots, n$ and $w > 0$. Moreover, set  $c \coloneqq \inf \funk  + \phi$. Then the following hold
\begin{enumerate}[topsep=0em,label=(\arabic*)]
    \item\label{rsg1} $\sum \partial_{\phi_i} \funk_i \subseteq \partial_{\sum \phi_i} \sum \funk_i $
    \item\label{rsg2} $w \partial_\phi \funk = \partial_{w \phi} (w \funk)$
    \item\label{rsg3} If $\funk$ is convex then $\partial_0 \funk (\signal) \subseteq \partial_\phi \funk(\signal)$ for any $\signal \in \X$
    \item\label{rsg4} $\partial_\phi \funk (\signal)$ is convex and (weak*) closed
    \item\label{rsg5} If $\signal_\plus \in \argmin  \funk(\signal)$ then $0 \in \partial_\phi \funk (\signal_\plus)$
    \item\label{rsg6} $0 \in \partial_\phi \funk (\signal) \Longleftrightarrow \funk(\signal) \leq c$
    \item\label{rsg7} If $\funk$ is Lipschitz  and $\phi$ bounded on bounded subsets, then $\partial_\phi \funk$ is bounded. In particular, in this case the set $\partial_\phi \funk$ is weak*-compact.
    \item\label{rsg8} Let $p_k = g_k + z_k$ where $g_k \in \partial_\phi \funk(\signal_k)$ and $\norm{z_k} \leq \varepsilon_k$ with $\varepsilon_k \to 0$. Assume that $\signal_k$ converge weakly to $\signal_\plus$ and $g_k$ converge to $g$ and that $\funk$ is weakly lower semi-continuous. Then $g \in \partial_\phi \funk(\signal_\plus)$. If, instead, $\signal_k$ converge strongly to $\signal_\plus$, $g_k$ converge weakly to $g$ and $\funk$ is lower semi-continuous then we also have $g \in \partial_\phi \funk(\signal_\plus)$.
\end{enumerate}
\end{lemma}

\begin{proof}\mbox{}
\ref{rsg1} Let $p_i \in \partial \funk_i(\signal)$ and define $p = \sum_{i} p_i$. Then we have
    \begin{equation*}
        \sum_i \funk_i(\signal) + \innerprod{p, u - \signal} = \sum_i \left( \funk_i(\signal) + \innerprod{p_i, u - \signal} \right)  \leq \sum_i \left( \funk_i(u) + \phi_i(u) \right)
    \end{equation*}
    and hence the claim follows.

\ref{rsg2} Assume that $p \in \partial_\phi \funk(\signal)$. Then we have $w \funk(\signal) + w \innerprod{p, u - \signal} \leq w \left( \funk(u) + \phi(u) \right)$ by non-negativity of $w$ and hence $w \partial_\phi \funk \subseteq \partial_{w \phi} (w \funk)$.
    Now let $p \in \partial_{w \phi} (w \funk)$ then we define $q = \frac{p}{w}$ and it follows
    \begin{equation*}
        w \left( \funk(\signal) + \innerprod{q, u - \signal} \right) = w \funk(\signal) + \innerprod{p, u - \signal}
        \leq w \left( \funk(u) + \phi(u) \right)
    \end{equation*}
    which shows that $p \in w \partial_\phi \funk(\signal)$.

 \ref{rsg3} This is an immediate consequence of $\funk(u) \leq \funk(u) + \phi(u)$ by non-negativity of $\phi$.

    \ref{rsg4} Let $p_1, p_2 \in \partial_\phi \funk(\signal)$ and $\lambda \in (0, 1)$. Then we have
    \begin{align*}
        \funk(\signal) + \innerprod{\lambda p_1 + (1 - \lambda) p_2, u - \signal} &= \lambda ( \funk(\signal) + \innerprod{p_1, u -\signal}) + (1 - \lambda) (\funk(\signal) + \innerprod{p_2, u -\signal}) \\
        &\leq \lambda (\funk(u) + \phi(u)) + (1-\lambda) (\funk(u) + \phi(u)),
    \end{align*}
    which proves the convexity of $\partial_\phi \funk$.
    Now let us assume that $p_k \in \partial \funk (\signal)$ with $p_k$ (weak*) converges to $p$. By (weak*) convergence we have $\innerprod{p_k, u - \signal} \to \innerprod{p, u - \signal}$ and hence $p$ is also an relatively sub-differentiable subgradient.

    \ref{rsg5}     This is also a consequence of the non-negativity of $\phi$ and the assumption that $\signal_\plus$ is a global minimizer.

    \ref{rsg6} Let $0 \in \partial_\phi \funk(\signal)$. Then by definition we have $\funk(\signal) \leq \funk(u) + \phi(u)$ for any $u \in \X$ and hence also $\funk(\signal) \leq c$. On the other hand, if $\funk(\signal) \leq c$ then we have $\funk(\signal) \leq \funk(u) + \phi(u)$ for any $u \in \X$ and hence $0 \in \partial_\phi \funk(\signal)$.

    \ref{rsg7} Let $p \in \partial_\phi \funk(\signal)$ and set $u = \signal + v$ with $\norm{v} = 1$. Using the defining inequality we find \begin{align*}
        \innerprod{p, v} \leq \funk(\signal + v) - \funk(\signal) + \phi(\signal + v) \leq L + \phi(\signal + v)
    \end{align*}
    and thus by taking the supremum over $v$ we find that $\norm{p}$ is bounded.
    Using Banach-Alaouglu we see that $\partial_\phi \funk$ must be weak*-compact.

      \ref{rsg8} By assumption $\signal_k$ is bounded. Thus, we have
      \begin{align*}
          \funk(\signal_\plus) + \innerprod{g, u - \signal_\plus} &\leq \liminf_k \funk(\signal_k) + \innerprod{p_k, u - \signal_k}\\
          &\leq \liminf_k \funk(u) + \phi(u) + \innerprod{z_k, u - \signal_k}\\
          &\leq \funk(u) + \phi(u) + \liminf_k \varepsilon_k \norm{u - \signal_k} \\
          &= \funk(u) + \phi(u),
      \end{align*}
      which proves the claim.
\end{proof}

Lemma~\ref{lemma:basic} gives us a characterization of $\phi$-critical  points  as points $\signal$ for which $\funk + \phi$ is an upper bound of  $\funk(x)$. This characterization in particular implies that for any differentiable and relatively sub-differentiable function $\funk$ we have that the points $\signal$ with $\funk'(\signal) = 0$ must have bounded value independent of $\signal$.  Comparing this to the convex case we have that $\signal$ is a critical  point of the function $\funk$ if and only if $\signal$ is a global minimizer. In some sense, the  definition of $\phi$-critical  points allows for some error to be made and guarantees that  $\phi$-critical points cannot have arbitrarily large $\funk$-value. Moreover, whenever $\funk$ is coercive then all $\phi$-critical points must be inside some ball $B_r(0)$ for some $r > 0$.

\subsection{Computation of $\phi$-critical  points}
\label{ssec:algo}

We next answer the question of how to obtain $\phi$-critical  points at least for the case where  $\X$ is a Hilbert space.  Clearly, if $\funk$ is differentiable then one could consider classical gradient descent methods. Since we are also interested in non-differentiable functions, gradient descent in its classical form may not be applicable. Below we show that  a generalized gradient method using relative subgradients instead of gradients will yield $\phi$-critical  points in the sense of Definition~\ref{def:criticalpoint}. This shows that Algorithm~\ref{alg:descent} is a natural extension of subgradient descent \cite{shor2012minimization, boyd2003subgradient}.

\begin{algorithm}
\caption{Relative subgradient descent} \label{alg:descent}
\begin{algorithmic}
\Require Starting point $\signal_0 \in \X$, stepsizes $\eta_n > 0$
\State $n \gets 0$
\While{$0 \notin \partial_\phi \funk (\signal_n)$}
\State Choose $g^*_n \in \partial_\phi \funk (\signal_n)$ and $g_n \in \X$ such that $\innerprod{g_n^*, g_n} > 0$
\State $\signal_{n+1} = \signal_n - \eta_n g_n$
\State $n \gets n + 1$
\EndWhile
\end{algorithmic}
\end{algorithm}

The following results shows that Algorithm~\ref{alg:descent} converges to a $\phi$-critical point of the function $\funk$.
The given proof closely follows the one given in \cite{boyd2003subgradient} but does not assume a finite dimensional setting and considers relatively sub-differentiable functionals instead of sub-differentiable function.

\begin{theorem}[Convergence of Algorithm~\ref{alg:descent}]
\label{thm:descent} Assume that $\X$ is a Hilbert space and that $\funk$ is relatively sub-differentiable with bound $\phi$. Moreover, choose  $g_n = \lambda_n g_n^*$ in Algorithm~\ref{alg:descent} with $\lambda_n > 0$ such that $\norm{g_n} \leq C$ for all $n \in \N$.
Then for any point $u \in \X$ and any step $N \in \N$ we have
\begin{align*}
    \min_{i=1, \dots, N} \funk(\signal_i) \leq \funk(u) + \phi(u) + \frac{\norm{\signal_0 - u}^2 + C^2 \sum_{i=1}^N \eta_i^2}{2 \sum_{i=1}^N \eta_i}.
\end{align*}
\end{theorem}

\begin{proof}
 Let $u \in \X$. After rescaling of $g_n^*$ according to assumption we may assume that $g_n = g_n^*$. Then by definition of $\signal_{n+1}$ we have
\begin{align*}
    \norm{\signal_{n+1} - u}^2 &= \norm{\signal_n - u}^2 - 2 \eta_n \innerprod{g_n, \signal_n - u} + \eta_n^2 \norm{g_n}^2 \\
    &\leq \norm{\signal_n - u}^2 - 2 \eta_n (\funk(\signal_n) - \funk(u) - \phi(u)) + \eta_n^2 \norm{g_n}^2.
\end{align*}
Applying this inequality recursively and using the fact that $\norm{\signal_{n+1} - u}^2 \geq 0$ we find
\begin{align*}
    2\sum_{i=1}^{n} \eta_i (\funk(\signal_i) - \funk(u) - \phi(u)) \leq \norm{\signal_0 - u}^2 + \sum_{i=1}^n \eta_i^2 \norm{g_i}^2,
\end{align*}
which together with the inequalities $\norm{g_i} \leq C$ and $\sum_{i=1}^n \eta_i \funk(\signal_i) \geq \min_{i=1,\dots,n} \funk(\signal_i) \sum_{i=1}^n \eta_i$ shows the desired result.
\end{proof}

Theorem \ref{thm:descent} shows that under the assumption that the sequence of step-sizes $(\eta_n)_{n \in \N}$ is square-summable but not summable, then in the limit we have $\lim_{n \to \infty} \min_{i=1, \dots, n} \funk(\signal_i) \leq \funk(u) + \phi(u)$ for any $u \in \X$.  Note the analysis and the proofs heavily rely on the usage of the functional $\phi$, but we note that at no point during Algorithm~\ref{alg:descent} do we need explicit knowledge of the functional $\phi$ but only access to elements of $\partial_\phi \funk$. In particular, in the case of Remark~\ref{re:differentiable} when using the gradient of $\funk$ as the update direction the generated sequence will yield a $\phi$-critical point.

Finally, assume that we have a functional of the form $\funk(\signal) = \similarity(\signal) + \al \reg(\signal)$ where each term is relatively sub-differentiable with bounds $\phi_\similarity$ and $\phi_\reg$. Then Lemma~\ref{lemma:basic} shows that for $s \in \partial_{\phi_\similarity} \similarity$ and $r \in \partial_{\phi_\reg} \reg$ we have $s + \al r \in \partial_{\phi_\similarity + \al \phi_\reg} \left( \similarity + \al \reg \right)$. This implies that Algorithm~\ref{alg:descent} can be applied in the case where we are looking for a $\phi$-critical  point of the sum of two relatively sub-differentiable functionals and only have access to elements of $\partial_{\phi_\similarity} \similarity$ and $\partial_{\phi_\reg} \reg$.

\section{Regularizing properties of $\phi$-critical points} \label{sec:theory}

In this section we present a  convergence analysis for $\phi$-critical points of  Tikhonov-type functionals $\tik_{\al, \data^\delta}$ extending the existing analysis  for global minima  \cite{scherzer2009variational}.
At this point we want to emphasize again that the assumption of being able to obtain global minima of $\tik_{\al, \data^\delta}$ can be extremely restrictive when $\reg$ is non-convex and the main goal of our analysis is to discard this assumption. Instead we focus only on $\phi$-critical  points of $\tik_{\al, \data^\delta}$ which may include local minimizers, saddle points or even local maxima.

Recall that  we are interested in Tikhonov-type  functionals $\tik_{\al, \data^\delta} \colon \X \to [0, \infty)$  of the form
\begin{equation} \label{eq:tik}
    \tik_{\al, \data^\delta}(\signal) = \similarity(\signal, \data^\delta) + \al \reg(\signal) \,,
\end{equation}
for given $\similarity \colon \X \times \Y \to [0, \infty)$ and $\reg \colon \X \to [0, \infty)$. Here, $\similarity$ is a similarity measure between $\signal$ and $\data^\delta$ and a standard situation we are interested in is  $\similarity(\signal, \data^\delta) = \frac{1}{2} \norm{\Ko (\signal) - \data^\delta}^2$ where $\Ko \colon \X \to \Y$ is the forward operator of the inverse problem of interest.  Instead of working with global minima of the functional \eqref{eq:tik} we consider regularized solutions $\signal_\al^\delta$ as $\al\phi$-critical  points of $\tik_{\al, \data^\delta}$, meaning
\begin{equation} \label{eq:crit-reg}
	0 \in \partial_{\al \phi} \left(\similarity(\cdot, \data^\delta) + \al \reg(\cdot) \right) (\signal_\al^\delta)
\end{equation}
We will analyze  stability and convergence of such critical points.

For the analysis we make the following assumptions.

\begin{cond}[Critical point regularization] \label{cond:a} \hfill
\begin{enumerate}[label=(C\arabic*), leftmargin=3em, topsep=0em, itemsep=0em]
\item $\X$ is a reflexive Banach spaces and $\Y$ is a metric space with metric $\distance$ \label{cond:a1}
\item $\reg$ is weakly sequentially lower semi-continuous  \label{cond:a2}

\item  $\reg$ is relatively sub-differentiable with bound $\phi$ \label{cond:a2b}
\item $\similarity$ is weakly sequentially lower semi-continuous, convex in its first argument and continuous in its second argument \label{cond:a3}
\item $\exists C > 0~\exists p \geq 1~\forall z \in \X~\forall \data, \data^\delta \in \Y \colon \similarity(z, \data) \leq C \left( \similarity(z, \data^\delta) + \distance(\data, \data^\delta)^p \right)$  \label{cond:a4}
\item $\forall \al > 0$ and $\forall \data^\delta \in \Y$ the functional $\tik_{\al, \data^\delta}$ is coercive, i.e. $\tik_{\al, \data^\delta}(\signal) \to \infty$ for $\norm{\signal} \to \infty$ \label{cond:a5}
\end{enumerate}
\end{cond}

Most of the assumptions in Condition~\ref{cond:a} are classical assumptions (or generalizations thereof), e.g. \cite{scherzer2009variational, obmann2021augmented, li2020nett}, made for the analysis of variational methods. For example, the coercivity assumption only poses a condition on the involved functionals ``at infinity'' and as such does not pose any form of condition, say for example, on the behaviour in a ball around $0$. This means, that the function $\tik$ can be highly non-convex as long as it is growing fast enough outside bounded sets.
The major difference in the analysis provided here is that $\reg$ is relatively sub-differentiable, which we have motivated in Section~\ref{sec:rsd}, and the assumption that in general the regularized solutions are not global minima but only $\phi$-critical  points.

One of the simplest and commonly used example of a similarity measure which satisfies Assumptions~\ref{cond:a3} and \ref{cond:a4} is given by $\similarity(\signal, \data^\delta) = \norm{\Ko \signal - \data^\delta}^p$ whenever $\Ko \colon \X \to \Y$ is the linear forward operator of the underlying inverse problem and $\Y$ is a Banach space. In general, any similarity measure of the form $\norm{\Lo (\Ko \signal - \data^\delta)}^p$ satisfies these assumptions, if $\Lo$ is a linear and bounded operator, e.g. a reweighting of the residual $(\Ko \signal - \data^\delta)$.

We now turn our focus to the stability and convergence analysis of the considered method, i.e. $\signal_\al^\delta \in \crit_{\alpha \phi} \tik_{\al, \data^\delta}$. We start with existence and stability results.

\subsection{Existence and stability}

\begin{theorem}[Existence] \label{thm:existence}
Under Assumption~\ref{cond:a} the problem is well-posed, i.e. for every $\al > 0$ and $\data^\delta \in \Y$ the set $\crit_{\alpha \phi} \tik_{\al, \data^\delta}$ is non-empty.
\end{theorem}
\begin{proof}
This is an immediate consequence of the existence of minimizers of $\tik_{\al, \data^\delta}$ which follows from the coercivity and the continuity assumptions on the functional $\tik_{\al, \data^\delta}$. A more detailed proof can be found in \cite{scherzer2009variational}.
\end{proof}

Clearly, $\al\phi$-critical  points may exist under weaker assumptions than a coercivity assumption. However, the coercivity is an important property in the following analysis which guarantees the existence of a weakly convergent subsequence whenever the sequence is bounded. As such we have also derived existence of $\phi$-critical  points using the coercivity. Extending the current analysis to the case of non-coercive functionals $\tik_{\al, \data^\delta}$ is subject to future work.

Another advantage of using $\phi$-critical points opposed to global minima, besides being numerically and hence practically more tractable for non-convex functionals, is that we have a simple way of talking about ``inexact'' critical points, i.e. points where the gradient is small but not necessarily $0$.
As it turns out, the following analysis can be performed under the even weaker assumption that the stabilized solutions are ``inexact'' critical points instead of exact critical points.

\begin{theorem}[Stability] \label{thm:stability}
Let $\data^\delta \in \Y, \al > 0$ and $\data_k \to \data^\delta$ and assume that $\signal_k \in \X$ is such that $z_k \in \partial_{\al \phi} \left( \similarity(\cdot, \data_k) + \al \reg(\cdot) \right) (\signal_k)$ with $\norm{z_k} \to 0$ and $\innerprod{z_k, \signal_k} \leq 0$. Then the sequence $(\signal_k)_k$ has a weakly convergent subsequence and the limit $\signal_\plus$ of every weakly convergent subsequence is an $\alpha \phi$-critical  point of $\tik_{\al, \data^\delta}$.
\end{theorem}
\begin{proof}
To show the existence of a weakly convergent subsequence, using the reflexivity of $\X$, it is enough to show that $(\signal_k)_k$ is a bounded sequence. By coercivity of $\tik_{\al, \data^\delta}$ it is enough to show that $(\tik_{\al, \data^\delta}(\signal_k))_k$ is bounded. We have for any $u \in \X$
\begin{align*}
    \similarity(\signal_k, \data_k) + \al \reg(\signal_k) + \innerprod{z_k, u - \signal_k} \leq \similarity(u, \data_k) + \al \reg(u) + \al \phi(u)
\end{align*}
and using $\innerprod{z_k, \signal_k} \leq 0$ it follows
\begin{align*}
    \similarity(\signal_k, \data_k) + \al \reg(\signal_k) \leq \similarity(u, \data_k) + \al \reg(u) + \al \phi(u) + \norm{z_k} \norm{u}.
\end{align*}
By assumption on $\similarity$ we have $\similarity(\signal_k, \data^\delta) \leq C(\similarity(\signal_k, \data_k) + \distance(\data_k, \data^\delta)^p)$ which yields
\begin{align*}
    \similarity(\signal_k, \data^\delta) + \al \reg(\signal_k) &\leq C \left( \similarity(\signal_k, \data_k) + \al \reg(\signal_k) + \distance(\data_k, \data^\delta)^p \right)\\
    &\leq C \left( \similarity(u, \data_k) + \al \reg(u) + \al \phi(u) + \norm{z_k} \norm{u} + \distance(\data_k, \data^\delta)^p \right) \\
    &\leq \tilde{C} \left( \similarity(u, \data^\delta) + \al \reg(u) + \al \phi(u) + \norm{z_k} \norm{u} + \distance(\data_k, \data^\delta)^p \right)
\end{align*}
for any $u \in \X$. By assumption we have $\norm{z_k} \to 0$ and $\distance(\data_k, \data^\delta) \to 0$ so the right hand side is bounded for $k$ large enough. This shows that there exists some weakly convergent subsequence.

Let now $(\signal_k)_k$ denote such a subsequence and denote by $\signal_\plus$ its limit. Using the weak lower semi-continuity of the involved functionals it follows for any $u \in \X$
\begin{align*}
    \similarity(\signal_\plus, \data^\delta) + \al \reg(\signal_\plus) &\leq \liminf_k \similarity(\signal_k, \data_k) + \al \reg(\signal_k) + \innerprod{z_k, u - \signal_k}\\
    & \leq \liminf_k \similarity(u, \data_k) + \al \reg(u) + \al \phi(u) + \norm{z_k} \norm{u}\\
    &= \similarity(u, \data^\delta) + \al \reg(u) + \al \phi(u)
\end{align*}
where the last equality follows from continuity of $\similarity$ in its second argument. This shows that $0 \in \partial_{\al \phi} \left( \similarity(\cdot, \data^\delta) + \al \reg(\cdot) \right) (\signal_\plus)$.
\end{proof}

Clearly, whenever $z_k = 0$, i.e. $\signal_k$ is an $\alpha \phi$-critical point, then the assumptions on $z_k$ in Theorem~\ref{thm:stability} are satisfied. It follows that $\alpha \phi$-critical  points are stable in the above sense. However, Theorem~\ref{thm:stability} also shows that we do not need access to exact $\alpha \phi$-critical  points but rather points which are in some sense close to an $\alpha \phi$-critical  point.

\begin{remark}[Inexact critical points obtained by use of minimization schemes] \label{re:inexact}
Consider once again the case of Remark~\ref{re:differentiable} and assume that the $\alpha \phi$-critical  points are obtained by using gradient descent or any other algorithm which finds zeros of the gradient. Then we have $z_k = \similarity'(\signal_k, \data_k) + \al \reg'(\signal_k)$ and whenever $\norm{z_k} \to 0$ and $\innerprod{z_k, \signal_k} \leq 0$ we have that the considered points have a weakly convergent subsequence.
For practical applications this means, that we have an easily verifiable condition which can be used as a kind of stopping criterion when searching for critical points. As a consequence, we do not have to guarantee that the regularized solutions are critical points but rather are ``close'' to being a critical point.
\end{remark}

\subsection{Convergence}

The next goal is to show the convergence of the regularized solutions to a solution of the original problem in the case that the noise-level $\delta$ tends to $0$. Here, we call $z \in \X$ an $\similarity$-solution of $\data \in \Y$ if $\similarity(z, \data) = 0$. Like in the case of Theorem~\ref{thm:stability}, the proof can be done under the weaker assumption of only having access to ``inexact'' $\phi$-critical points (see Remark~\ref{re:inexact}).

\begin{theorem}[Convergence] \label{thm:convergence}
Let $\data \in \Y$ and assume it has an $\similarity$-solution. Further, let $\data_k \in \Y$ with $\distance(\data_k, \data) \leq \delta_k$ with $\delta_k \to 0$. Choose $\al = \al(\delta)$ such that for $\al_k = \al(\delta_k)$ we have $\lim_k \al_k = \lim_k \delta_k^p / \al_k = 0$. Assume that the regularized solutions $\signal_k \in \X$ are such that $z_k \in \partial_{\al_k \phi} \left( \similarity(\cdot, \data_k) + \al_k \reg(\cdot) \right) (\signal_k)$ with $\norm{z_k} / \al_k \to 0$ and $\innerprod{z_k, \signal_k} \leq 0$. \\
Then the sequence $(\signal_k)_k$ has a weakly convergent subsequence and the limit $\signal_\plus$ of any such sequence is an $\similarity$-solution of $\data$. Moreover, we have $\reg(\signal_\plus) \leq \reg(u) + \phi(u)$ for any $\similarity$-solution $u$.
Finally, whenever the $\similarity$-solution is unique then $(\signal_k)_k$ converges weakly to this solution.
\end{theorem}

\begin{proof}
Similar to the stability proof we show that $(\signal_k)_k$ is bounded by using the coercivity of the functionals $\tik_{\al, \data^\delta}$. Following the above proof we find for any $u \in \X$
\begin{align*}
    \similarity(\signal_k, \data_k) + \al_k \reg(\signal_k) \leq \similarity(u, \data_k) + \al_k \reg(u) + \al_k \phi(u) + \norm{z_k} \norm{u}
\end{align*}
and by choosing $u$ such that $\similarity(u, \data) = 0$ we find $\similarity(u, \data_k) \leq C \delta_k^p$ which implies
\begin{align*}
    \similarity(\signal_k, \data_k) + \al_k \reg(\signal_k) \leq C \delta_k^p + \al_k (\reg(u) + \phi(u)) + \norm{z_k} \norm{u}.
\end{align*}
Since both $\similarity$ and $\reg$ are non-negative it then follows
\begin{align*}
    \lim_k \similarity(\signal_k, \data_k) &= 0\\
    \limsup_k \reg(\signal_k) &\leq \reg(u) + \phi(u),
\end{align*}
where we have used the assumptions $\lim_k \delta_k^p / \al_k = \lim_k \norm{z_k}/ \al_k = 0$. This shows that $(\reg(\signal_k))_k$ is a bounded sequence and using once again $\similarity(\signal_k, \data) \leq C \left( \similarity(\signal_k, \data_k) + \delta_k^p \right)$ we find that for $\al^\plus = \max \set{\al_k \colon k \in \N}$ the sequence $(\similarity(\signal_k, \data) + \al^\plus \reg(\signal_k))_k$ is bounded. Using the coercivity of $\similarity(\cdot, \data) + \al^\plus \reg(\cdot)$ we get that the sequence $(\signal_k)_k$ is bounded and hence has a weakly convergent subsequence.

Finally, using the weak lower-semicontinuity of $\similarity$ and $\reg$ we have that for any such weakly convergent subsequence with limit $\signal_\plus$
\begin{align*}
    \similarity(\signal_\plus, \data) &\leq \liminf_k \similarity(\signal_k, \data_k) = 0\\
    \reg(\signal_\plus) &\leq \liminf_k \reg(\signal_k) \leq \reg(u) + \phi(u)
\end{align*}
for any $u \in \X$ with $\similarity(u, \data) = 0$.

Whenever the solution is unique, then every subsequence of $(\signal_k)_k$ has a subsequence converging to this solution. This shows that $(\signal_k)_k$ converges weakly to the unique solution.
\end{proof}

At this point, we want to emphasize once again, that the assumptions on the choice of points $\signal_k$ in Theorem~\ref{thm:convergence} are weaker than the assumption that $\signal_k$ is an $\alpha_k \phi$-critical point and that in particular the analysis also holds for these points.

Since for this section we only assume that $\reg$ is relatively sub-differentiable without explicit knowledge of the bound $\phi$ Theorem~\ref{thm:convergence} gives a somewhat intangible condition on the type of solutions we obtain in the limit $\delta \to 0$. A more tangible condition, and more importantly one independent of $\phi$, is given by the next theorem, where we assume a separability condition on the gradients $z_k \in \partial_{\al_k \phi} \left( \similarity(\cdot, \data_k) + \al_k \reg(\cdot) \right) (\signal_k)$.
This separability assumption can be satisfied in many cases, e.g. when $\similarity$ and $\reg$ are (relatively sub-)differentiable and the $\phi$-critical points arise due to some algorithm such as gradient descent. Using these algorithms we are often in the situation that $z_k = s_k + \al_k r_k$ where $s_k$ is a (sub-)gradient of $\similarity$ and $r_k$ is an (relatively sub-differentiable sub-)gradient of $\reg$. Assuming that the gradients $(r_k)_k$ of $\reg$ have a cluster point, we get the additional following property of these cluster points.

\begin{theorem}[Normality property of the solution] \label{thm:normality}
Let the same assumptions as in Theorem~\ref{thm:convergence} hold and denote by $(\signal_k)_k$ a weakly convergent subsequence with limit $\signal_\plus$. Let $z_k = s_k + \al_k r_k$ where $s_k \in \partial_0 \similarity(\cdot, \data_k) (\signal_k)$ and $r_k \in \partial_\phi \reg (\signal_k)$. \\
Then any cluster point $r$ of the sequence $(r_k)_k$ satisfies $-r \in -\partial_\phi \reg (\signal_\plus) \cap N_{L(\data)}(\signal_\plus)$, where $N_{L(\data)}(\signal_\plus)$ is the normal cone of the convex set of all $\similarity$-solutions of $\data$ at $\signal_\plus$.
\end{theorem}

\begin{proof}
Let $r$ be a cluster point of the sequence $(r_k)_k$. Then by weak lower semi-continuity of $\reg$ and by assumption on $r_k$ we have
\begin{align*}
    \reg(\signal_\plus) + \innerprod{r, u - \signal_\plus} &\leq \liminf_k \reg(\signal_k) + \innerprod{r_k, u - \signal_k} \leq \reg(u) + \phi(u),
\end{align*}
which shows that $r \in \partial_\phi \reg (\signal_\plus)$.

Now assume that $u \in \X$ is such that $\similarity(u, \data) = 0$. Then we have $\similarity(u, \data_k) \leq C \delta_k^p$ and it follows
\begin{align*}
    \innerprod{-r, u - \signal_\plus} &= \lim_k \innerprod{\frac{s_k - z_k}{\al_k}, u - \signal_k} \\
    &\leq \lim_k \frac{1}{\al_k} \left( \similarity(u, \data_k) - \similarity(\signal_k, \data_k) + \norm{z_k} \norm{u} + \innerprod{z_k, \signal_k} \right)\\
    &\leq \lim_k \frac{1}{\al_k} \left( \similarity(u, \data_k) + \norm{z_k} \norm{u}  \right) \\
    &\leq \lim_k C \frac{\delta_k^p}{\al_k} + \norm{u} \frac{\norm{z_k}}{\al_k}\\
    &= 0,
\end{align*}
where we have used the convexity of $\similarity$ in its first argument and the assumption on the limits of the sequences $(\al_k)_k$ and $\left( \norm{z_k} / \al_k \right)_k$.
\end{proof}

Theorem~\ref{thm:normality} shows that the solution $\signal_\plus$ obtained by critical point regularization satisfies some form of first order optimality conditions, see e.g. \cite{rockafellar2015convex}.

Also note that in the case where $\reg$ is convex and we choose $\phi = 0$, both properties in Theorem~\ref{thm:convergence} and \ref{thm:normality} reduce to the common property that $\signal_\plus$ is an $\reg$-minimizing solution, i.e. $\reg(\signal_\plus) \leq \reg(u)$ for any $u \in \X$ with $\similarity(u, \data) = 0$.

\begin{remark}[Convex regularizers]
Clearly, any sub-differentiable convex function is relatively sub-differentiable with the choice $\phi = 0$. Nevertheless, one could also choose $\phi = \varepsilon > 0$. With this choice we see that the results in Theorem~\ref{thm:convergence} roughly state that the solutions $\signal_\plus$ we approximate by using critical point regularization are $\reg$-minimizing solutions up to an error $\varepsilon$ whenever the regularized solutions are minimizers up to an error of $\varepsilon$. \\
This result, as opposed to classical variational regularization theory e.g. \cite{scherzer2009variational}, has the advantage that at no point do we require exact global minimizers of the functionals $\tik_{\al, \data}$ but only approximate minimizers, which may be more easily reachable in practical applications.
Consider for example the case where we employ an iterative algorithm which has convergence guarantees of the form $\funk(\signal_n) - \funk(\signal_*) \leq C / n^r$ for the $n$-th iterate and $\signal_*$ being a minimizer of $\funk$. Applying this algorithm to $\funk = \tik_{\al, \data}$ and requiring that $C / n^r \leq \alpha \varepsilon$ in order to get that $\signal_n$ is an $\alpha \phi$-critical point, we see that the above theory shows that one might stop the iterative algorithm after a finite amount of steps, i.e. we do not necessarily need to run the algorithm until it converges and we still get a stable and convergent regularization method.\\
At first it might seem that a disadvantage of this is that we do not achieve an $\reg$-minimizing solutions in the limit. However, this can also be circumvented by considering a variable $\varepsilon$. To be more precise, following the proof of Theorem~\ref{thm:convergence} with $\varepsilon = \varepsilon(\delta)$ and the condition $\varepsilon(\delta) \to 0$ as $\delta \to 0$, it is easy to see that to obtain a sequence $(\signal_k)_k$ weakly converging to an $\reg$-minimizing solution it is enough to run the iterative algorithm for a number of iterations steps $n_k$ such that $C / n_k^r \leq \alpha_k \varepsilon_k$.
\end{remark}

We next discuss another special case of our analysis which pertains to functionals such as the one in Remark~\ref{re:examples}.

\subsection{Differentiable regularizers and classical critical points}

In this subsection we consider the important special case where the $\phi$-critical points are given by classical critical points, i.e. by points $\signal$ for which $\tik_{\al, \data}'(\signal) = 0$ and we give stability and convergence results for this case. To this end, we assume that the bound $\phi$ can be constructed in such a way that $\reg(\signal) + \innerprod{\reg'(\signal), u - \signal} \leq \reg(u) + \phi(u)$, see e.g. Remark~\ref{re:differentiable}. Then, if $\similarity$ is differentiable in its first argument, by convexity of $\similarity$, we have $\similarity'(\signal, \data^\delta) + \al \reg'(\signal) \in \partial_{\al \phi} (\similarity(\cdot, \data^\delta) + \al \reg(\cdot))(\signal)$. This shows, that whenever we employ some algorithm which finds a classical critical point, we also obtain an $\alpha \phi$-critical point in the sense of Definition~\ref{def:errorconvex} which satisfies the separability assumption necessary for Theorem~\ref{thm:normality}. In particular, these points are amenable to the analysis above.

Nevertheless, the analysis relies on an abstract concept of $\phi$-critical points and even in the case where the involved functionals are differentiable we cannot guarantee that the limits will again be $\phi$-critical points without any additional assumptions. In order to guarantee this we need the assumption that $\similarity'$ and $\reg'$ are weakly (sequentially) continuous. Combining the above theorems we then get the following result.

\begin{proposition}[Existence, stability and convergence for classical critical points] \label{prop:convergence}
Assume that $\similarity$ and $\reg$ are differentiable with weakly continuous derivatives and let Condition~\ref{cond:a} hold. Moreover, let $\data, \data^\delta \in \Y$ and $\al > 0$ and assume that $\data$ has an $\similarity$-solution. Then the following hold
\begin{enumerate}
    \item \textbf{Existence:} $\tik_{\al, \data^\delta}$ has at least one $\phi$-critical  point.
    \item \textbf{Stability:} If $(\data_k)_k \subseteq \Y$ is a sequence converging to $\data^\delta$ and $\signal_k$ is such that $z_k = \similarity'(\signal_k, \data_k) + \al \reg'(\signal_k) \to 0$ as $k \to \infty$ and $\innerprod{z_k, \signal_k} \leq 0$. Then $(\signal_k)_k$ has a weakly convergent subsequence and any weak clusterpoint $\signal_\plus$ of $(\signal_k)_k$ is a critical point of $\tik_{\al, \data^\delta}$.
    \item \textbf{Convergence:} Let $(\data_k)_k \subseteq \Y$ be a sequence with $\distance(\data_k, \data) \leq \delta_k$ and $\al = \al(\delta)$ be such that for $\al_k = \al(\delta_k)$ we have $\lim_k \al_k = \lim_k \delta_k^p / \al_k = 0$. Then, if we choose $\signal_k \in \X$ such that $z_k = \similarity'(\signal_k, \data_k) + \al_k \reg'(\signal_k)$ satisfies $\lim_k \norm{z_k}/ \al_k = 0$ and $\innerprod{z_k, \signal_k} \leq 0$ the sequence $(\signal_k)_k$ has at least one weak clusterpoint and any such clusterpoint $\signal_\plus$ is an $\similarity$-solution of $\data$ with the following additional properties
    \begin{enumerate}
        \item $\reg(\signal_\plus) \leq \inf_{\similarity(u, \data) = 0} \reg(u) + \phi(u)$
        \item $\innerprod{-\reg'(\signal_\plus), z - \signal_\plus} \leq 0$ for any $z \in \X$ with $\similarity(z, \data) = 0$, i.e. $-\reg'(\signal_\plus) \in N_{L(\data)}(\signal_\plus)$.
    \end{enumerate}
    Finally, whenever the $\similarity$-solution is unique then $(\signal_k)_k$ converges weakly to this solution.
\end{enumerate}
\end{proposition}

\begin{proof}
This follows immediately by applying Theorems~\ref{thm:existence},~\ref{thm:stability}, \ref{thm:convergence} and \ref{thm:normality}.
\end{proof}

For the differentiable case this identifies the limiting problem we solve by regularizing the inverse problems with critical points, i.e. in the limit we find solutions which satisfies a first order optimality condition of the constrained optimization problem
\begin{align*}
    \inf_u \reg(u) \quad \textrm{such that} \quad \similarity(u, \data) = 0.
\end{align*}


We now briefly discuss the special case where $\similarity$ is given as the norm-discrepancy, e.g. in the case where $\Y$ is a Hilbert-space.

\begin{lemma}[Solution for norm discrepancy] \label{lemma:convergence}
Let the same assumptions as in Proposition~\ref{prop:convergence} hold and assume that $\similarity(\signal, \data^\delta) = \frac{1}{p} \norm{\Ko \signal - \data^\delta}_\Y^p$ for some $p > 1$ where $\Ko \colon \X \to \Y$ is a linear and bounded forward operator between Banach spaces and $\norm{\cdot}_\Y^p$ is differentiable. Furthermore, denote by $\signal_\plus$ a solution according to Proposition~\ref{prop:convergence}. \\
Then we have $-\reg'(\signal_\plus) \in \ker(\Ko)^\perp = \set{p \in \X^\ast \colon \forall \signal_0 \in \ker(\Ko) \colon \innerprod{p, \signal_0} = 0}$.
\end{lemma}

\begin{proof}
Any solution $z$ can be written as $z = \signal_\plus + \signal_0$ where $\signal_0 \in \ker(\Ko)$. By using $x_0$ and $-x_0$, Proposition~\ref{prop:convergence} shows that $\innerprod{-\reg'(\signal_\plus), \signal_0} = 0$ for any $\signal_0 \in \ker(\Ko)$ and hence the claim follows.
\end{proof}


\section{Examples and comparison}
 \label{sec:discussion}

In this section we compare the proposed regularization concept using   $\al \phi$-critical points to standard  Tikhonov regularization, its  convex relaxation and discuss the influence of different choices for $\phi$.  Moreover, we discuss ReLU-regularizers as a class of  non-convex  regularizers for which the presented theory is applicable.

\subsection{Dependence on the choice of  $\phi$}
\label{ssec:convex}

We begin this section with a comparison between classical Tikhonov regularization and  regularization with $\phi$-critical points using different choices for $\phi$. For the sake of simplicity we consider a convex example, where the data-fidelity is chosen as $\norm{\Ko \signal - \data_\delta}^2/2$ and the regularizer is given by $\reg(\signal) =  \norm{\signal}^2/2$. We assume that $\al > 0$ and $\data_\delta \in \Y$ are fixed and denote by $\signal_{\al, \delta}$ the unique minimizer of the Tikhonov functional $ \tik_{\al, \data_\delta}(\signal) =  \norm{\Ko \signal - \data_\delta}^2/2 + \al  \norm{\signal}^2/2$. Note that in the convex case, standard  Tikhonov regularization corresponds to regularization with $\al\phi$-critical points for the choice $\phi \equiv 0$.

Consider now the case where $\phi \equiv \varepsilon > 0$  is constant.  Then, according to the general theory, any $\al\phi$-critical point $\signal$ of $\tik_{\al, \data_\delta}$  is characterized by
\begin{align*} 
    \tik(\signal)
    \leq \min_{u} \tik_{\al, \data_\delta}(u) + \al \phi(u)
    = \min_{u} \tik_{\al, \data_\delta}(u) + \al \varepsilon
    =  \tik_{\al, \data_\delta}(\signal_{\al, \delta}) + \al \varepsilon\,.
\end{align*}
Writing $\signal = \signal_{\al, \delta} + \signal_0$ we find after some rearrangements  that this is in turn equivalent to 
\begin{equation*} 
  \innerprod{\signal_0, r(\al)} + \frac{1}{2} \innerprod{\signal_0, L(\al) \signal_0} \leq \al \varepsilon \,,
  \end{equation*}
   where $L(\al) \coloneqq  (\Ko^* \Ko + \al I)$ and $r(\al) \coloneqq  L(\al) \signal_{\al, \delta} - \Ko^* \data_\delta$. Since $L(\al)$ is positive definite, this shows that $\signal_0$ has to be chosen in an ellipsoid around $0$ and thus $\signal = \signal_{\al, \delta} + \signal_0$ is contained in some ellipsoid around $\signal_{\al, \delta}$. Here, the size of the ellipsoid depends on the choice $\varepsilon$ and, for example, choosing $\varepsilon = 0$ leads to the singleton $\set{\signal_{\al, \delta}}$.

Let now $\phi$ be  an arbitrary non-negative function such that minimization of $\tik_{\al, \data_\delta}  + \al \phi $ is well-posed with minimizer   $\signal_\phi$. Then, following the  steps above, we find that  $\al\phi$-critical points $\signal = \signal_\phi + \signal_0$ of $\tik_{\al, \data_\delta} $ are characterized by
\begin{equation}\label{eq:phiconvex}
    \innerprod{\signal_0, r(\al)} + \frac{1}{2} \innerprod{\signal_0, L(\al) \signal_0} \leq \al \phi(\signal_\phi),
\end{equation}
where $L(\al)$ is as above and $r(\al) = L(\al) \signal_\phi - \Ko^* \data_\delta$. That is,  the set of $\al\phi$-critical points  is  an ellipsoid around the point $\signal_\phi$ where the size of the ellipsoid depends on the value $\phi(\signal_\phi)$. Depending on the choice of $\phi$ this value can even be $0$. To see this, consider for example the case where $\X = \ell^2(\N)$ and $\phi(\signal) = \beta \,  \norm{\max \set{0, \signal}}^2/2$ for some $\beta > 0$, where the maximum is taken pointwise. Then, if $\signal_{\al, \delta} \leq 0$ pointwise, we find that $\phi(\signal_{\al, \delta}) = 0$ and we arrive at $\signal_\phi = \signal_{\al, \delta}$ and  \eqref{eq:phiconvex} collapses to $\signal_0 = 0$.   In a similar fashion, the  choice $\phi(\signal) = \beta   \, \norm{\max(0, -\signal)}^2 / 2$ for some $\beta > 0$ leads to a point estimate whenever $\signal_{\al, \delta}$ is non-negative and to an ellipsoid if $\signal_{\al, \delta}$ has at least one negative entry.

 While we allow for the whole set of solutions defined as $\al\phi$-critical points of $\tik_{\al, \data_\delta}$, in practical applications one typically only chooses a  specific subset. More precisely, as shown in Section~\ref{ssec:algo} one can apply variants of gradient descent to construct critical points. In the convex situation this might lead to the same solution independent of the choice of $\phi$ since the classical gradient $\partial_0 \tik_{\al, \data_\delta}$ can be used as an $\al\phi$ subgradient for any choice  of $\phi$.  However, in the non-convex case using such algorithms will not find solutions which are global minima but only critical points as we discuss in the following subsection.

\subsection{Relation to convex  relaxation}

Next we investigate the relation between critical point regularization  for  a non-convex  example. Consider  $\X = \Y = \ell^2(\N)$ and $\Ko \signal = (k_i \signal_i)_{i\in \N}$. As non-convex regularizer we use a slightly perturbed double-well potential  $\reg(\signal) = \sum_{i} r_i(\signal_i)$, where for $q \in [1/2, 1)$ and $w_i>0$ we define
\begin{equation*}
    r_i(t) = \begin{cases}
     t^2/2   & \text{ for } t \leq q w_i\\
    (t - w_i)^2 / 2 + (q-1/2) w_i^2  &  \text{ for } t > q w_i \,.
    \end{cases}
\end{equation*}
Here, $q$ controls the perturbation away from $0$ of the local minimum at $w_i$. An illustration of $r_i$ for $w_i=1$ and two different values of $q$  is depicted in Figure~\ref{fig:doublewell}. While at this point the choice of the regularizer might seem somewhat arbitrary and contrived, we will argue later that the regularizer chosen here is a simplified model for a reasonable class of learned regularizers; see the discussion on ReLU-networks in Subsection \ref{sec:relu} and in particular Remark~\ref{re:multiwell}.

\begin{figure}[htb!]
    \centering
    \includegraphics[scale=0.5]{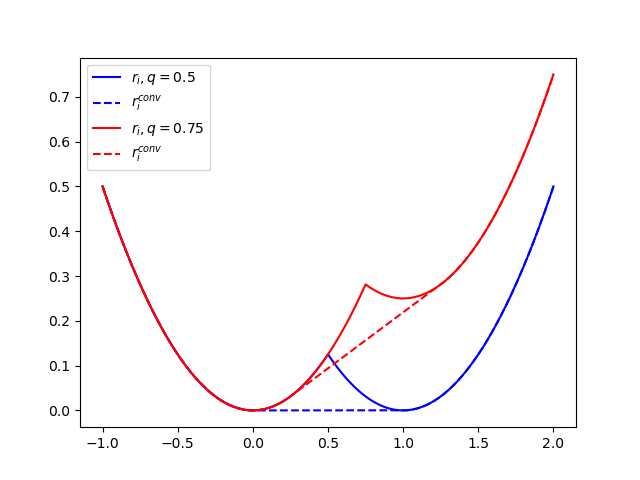}
    \caption{Illustration of the non-convex double-well regularizer (solid lines) for two different values of $q$ and the corresponding convex hull (dashed lines).}
    \label{fig:doublewell}
\end{figure}

For given $\al > 0, \data_\delta \in \Y$  we consider (classical) critical points of the Tikhonov functional $\tik_{\al, \data_\delta}$ defined by  
\begin{equation} \label{eq:critpoint}
    \Ko^*(\Ko \signal - \data_\delta) + \al \nabla \reg(\signal) = 0 \,.
\end{equation}
Although $r_i$ is not differentiable at $t = q w_i$, with slight abuse of notation we will denote by $r_i'(t)$ the derivative of $r_i$, which we define as $q w_i$ at $t = q w_i$. In the following the gradient $\nabla \reg (\signal) = (r_i'(\signal_i))_{i \in \N} $ will be understood with this convention.

\begin{remark}
One  can construct a specific $\phi$ such that  $\nabla \reg(\signal) \in \partial_{\al\phi} \reg (\signal)$. It is then guaranteed that  any classical critical point  is an $\al\phi$-critical point and hence a regularized solution fitting to the theory presented in this paper.  
Such a $\phi$ can be constructed by considering the defining inequality pointwise and constructing $\phi_i$ for $r_i$. The calculations for $\phi_i$ are relatively simple but tedious. Since the exact form of $\phi$ is irrelevant for our purpose, we refrain from explicitly defining   $\phi$. Instead, we focus on solutions that arise when solving equation~\eqref{eq:critpoint}. Further note that the algorithm presented in the Section~\ref{ssec:algo} with $\nabla \reg$ as relative subgradient is the same as classical gradient descent, thus further substantiating the assumption that the constructed regularized solution is a solution of~\eqref{eq:critpoint}.  \end{remark}

By definition of the operator and the regularizer, solutions of~\eqref{eq:critpoint} can be computed component-wise via $\forall i \colon k_i (k _i \signal_i - \data_i) + \al r_i'(\signal_i) = 0$. Hence 
\begin{align*}
        \signal_i = 
        \begin{cases}
        \frac{k_i \data_i}{k_i^2 + \al} 
        & \text{ for } k_i \neq 0 \wedge \frac{k_i \data_i}{k_i^2 + \al} \leq q w_i, 
        \\
        \frac{\al w_i + k_i \data_i}{k_i^2 + \al} 
        & \text{ for }  k_i \neq 0   \wedge \frac{k_i \data_i}{k_i^2 + \al} > q w_i
        \\ 
        \signal_i \in \set{0, w_i}
        & \text{ for }  k_i = 0 \,.
    \end{cases}
\end{align*}
An interesting observation is that for $k_i = 0$ we have a choice between $0$ and $w_i$ where the choice $\signal_i = w_i$ will, in general, only lead to a local minimizer instead of a global one; compare with Figure~\ref{fig:doublewell} for $q = 0.75$. In essence, this shows that the regularized solutions one obtains are a subset of the solutions of equation~\eqref{eq:critpoint} and need not be global minimizers.

Next we consider regularization with the convex relaxation of $\reg$, i.e. the convex hull $\reg^{\rm conv}$. After some lengthy calculations, which we do not present here for the sake of brevity, we derive for the convex hull of $r_i$ the form
\begin{align*}
    r_i^\text{conv}(t) = \begin{cases}
        t^2/2 & 
        \text{ for } t \leq (q-1/2) w_i, \\
        (t - w_i)^2/2 + (q-1/2) w_i^2 
        & \text{ for } t \geq (q + 1/2)w_i \\
        t(q - 1/2) - (q - 1/2) ^2 w_i^2/2 &
        \text{ otherwise }.
    \end{cases}
\end{align*}
A comparison of $r_i$ and $r_i^\text{conv}$ for two choices of the parameter $q$ is given in Figure~\ref{fig:doublewell} where the convex hull in both cases is visualized with a dashed line.

Solving the critical point equation~\eqref{eq:critpoint} where $\reg$ is replaced by its convex hull we find that the solutions differ quite a bit, at least for the case $k_i = 0$. 
Considering this case, we find that there are two different cases. First, whenever $q = 1/2$ the convex hull allows for arbitrary solutions $\signal_i \in [0, w_i]$. On the other hand if $q > 1/2$ the convex hull forces the choice $\signal_i = 0$. This means that even the slightest perturbation of the value of $r_i(w_i)$ will lead to a convex hull which loses this information. Comparing this to the solutions above we see that in the non-convex case we can always choose $\signal_i \in \set{0, w_i}$ independent of the value of $q$. This shows that there is a difference between the regularized solutions one obtains when considering the convex hull of the regularizer $\reg^{\rm conv}$ and the original regularizer $\reg$.

\subsection{Non-equivalence to Tikhonov regularization}

One might conjecture  that the proposed  regularization with $\al\phi$-critical points of $\tik_{\al, \data_\delta}$ is equivalent to Tikhonov regularization for some other modified choice  of the regularizer. While we cannot give a definite answer to this question at this point, we conjecture  that this is not the case. 

To support our hypothesis, let us analyze what would happen if the construction of $\al\phi$-critical points of $\tik_{\al, _\delta}$ were equivalent to the Tikhonov regularization with some regularizer $\reg_\phi$. In this case, the limiting problems would also coincide and thus for any $\data \in \ran(\Ko)$ and $\signal \in \X$ we have
\begin{equation*}
    \reg(\signal) \leq \min_{\Ko u = \data} \reg(u) + \phi(u) \Longleftrightarrow \reg_\phi(\signal) \leq \min_{\Ko u = \data} \reg_\phi(u) \,.
\end{equation*}
Denoting by $\signal_\phi$ a solution of $\min_{\Ko u = \data} \reg(u) + \phi(u)$ we have $\reg(\signal) \leq \reg(\signal_\phi) + \phi(\signal_\phi)$. Clearly, this is the case if and only if $\max \set{\reg , \reg(\signal_\phi) + \phi(\signal_\phi)}$ is minimal  among all possible solutions.
This essentially means that $\reg_\phi = \max \set{\reg , \reg(\signal_\phi) + \phi(\signal_\phi)}$.  

While such a choice can theoretically be used to characterize the limiting problem, we do not have access to $\signal_\phi$ and therefore cannot work with $\reg_\phi$ in practice. A slightly more subtle problem is that $\reg_\phi$ depends on the exact data and as such cannot be used for the case of noisy case where $\al > 0$. Thus, we conjecture that the proposed regularization is not equivalent to Tikhonov regularization independent of the choice $\reg_\phi$.

\subsection{ReLU-Networks as class of possible regularizers}
\label{sec:relu}

Next, we demonstrate that ReLU networks form a class of non-convex, relatively subdifferentiable regularizers that fit within the theory presented in this paper.   As discussed in \cite{li2020nett, lunz2018adversarial}, such regularizers are a powerful tool in the context of classical variational regularization.

Let now $\V$, $\U$ be further Hilbert spaces.  

\begin{definition}[Quasi-homogeneity] \label{def:quasi}
A function $f \colon \X \to \V$ is quasi-homogeneous, if there exists  $\Lc_f \colon \X \to L(\X, \V)$  such that  
$\sup_{\signal \in \X} \norm{\Lc_f(\signal)} < \infty$ and 
    $\sup_{\signal \in \X} \norm{f(\signal) - \Lc_f(\signal) \signal}  < \infty$. We call $\Lc_f$ the quasi-derivative of $f$.
\end{definition}

In Definition~\ref{def:quasi} and below $L(\X, \V)$ denotes the space of all bounded linear mappings from $\X$ to $\V$.   Quasi-homogeneity satisfies the following elementary rules.

\begin{lemma}[Quasi-homogeneity and relatively sub-differentiability]  \label{lem:quasi}
Let $f,h \colon \X \to \V$ and $g \colon \V \to \U$ be quasi-homogeneous, $\Wo  \in L(\X, \V)$ and $c \in \R$. Moreover, let  $v \in \V$ and  let $\psi \colon \V \to [0, \infty)$ be convex and sub-differentiable with $\psi(v) \leq C \norm{v}^p$ for some $C > 0$,  $p \geq 1$ and subgradient selection  $\psi'(v) \in (\partial_0 \psi)(v)$.  Then for some $\phi^1, \phi^2$ the following hold:
\begin{enumerate}[topsep=0em,label=(\arabic*)]
    \item $g \circ f$ is quasi-homogeneous with $\Lc_{g \circ f}(\signal) = \Lc_g(f(\signal)) \circ \Lc_f(\signal)$.
    \item $f + c h$ is  quasi-homogeneous with $\Lc_{f + c h} = \Lc_f + c \Lc_h$.
    \item $\signal \mapsto \Wo \signal + b$ is quasi-homogeneous.
\item   $\innerprod{v, f}$ is $\phi^1$-relative sub-differentiable  with  $\innerprod{v, \Lc_f(\signal) (\cdot)} \in \partial_{\phi^1} (\innerprod{v, f})(\signal)$. 
\item $\psi \circ f$ is $\phi^2$-relatively sub-differentiable with and  $\psi'(f(\signal))  \Lc_f(\signal) \in \partial_{\phi^2}(\psi \circ f)(\signal)$. 
\end{enumerate}
\end{lemma}

\begin{proof}
These properties follow immediately from the triangle inequality and the defining properties of quasi-homogeneity and relatively sub-differentiability. 
\end{proof}

\begin{theorem}[Learned regularizers] 
\label{thm:network}
Let $\psi \colon \V \to [0, \infty)$ be convex  and sub-differentiable with  $\psi(v) \leq C \norm{v}^p$ for some $C > 0$ and $p \geq 1$. Let $\act_\ell$  be quasi-homogeneous and $\Ao_\ell$ be affine and  continuous for $\ell \in \set{1, \dots, L}$. Then  
\begin{equation} \label{eq:network1}
     \net  = \act_L \circ \Ao_L \circ \dots \circ \act_1 \circ \Ao_1 
\end{equation}
is  quasi-homogeneous. Additionally,  $\reg  = \psi \circ \net $ is  relatively sub-differentiable.
\end{theorem}

\begin{proof}
Follows from repeated application of Lemma~\ref{lem:quasi}. 
\end{proof}

The crucial assumption in Theorem~\ref{thm:network} is that the activation functions $\act_\ell$ are quasi-homogeneous. This property is, for example,  satisfied in the case of the ReLU as the choice for the activation function as we discuss in the following example.

\begin{example}[ReLU regularizer]
Consider the case $\act_\ell = \relu$ defined by $\relu(\signal) = \max \set{0, \signal}$ where $\max$ is to be understood pointwise. Then $\Lc_{\relu}(\signal) = M_{g(\signal)}$, where $M_{g(\signal)}$ denotes pointwise multiplication with $g(\signal) = 0$ if $\signal \leq 0$ and $g(\signal) = 1$ if $\signal > 0$ and $g$ is again understood pointwise. The ReLU function is then quasi-homogeneous whenever the space $\X$ has the following property: For any $\signal, v  \in \X$ we have $\relu(\signal) \in \X, g(\signal) v \in \X$ and $\norm{g(\signal) v} \leq \norm{g(\signal)}_\infty \norm{v}$. Examples of such spaces are  $\X = \ell^r(\Lambda, L^p(\Omega, \mu))$ for some at most countable set $\Lambda$ and parameters $p, r \geq 1$. In particular, this also holds in the finite dimensional case  $\X = \R^n$. Thus, Theorem~\ref{thm:network} shows that ReLU networks are an appropriate choice to construct regularizers. 
We note here, that the same also holds true when the ReLU activation functions are replace by the more general class of parametric ReLU activation functions. 
\end{example}

Theorem~\ref{thm:network}  and Lemma~\ref{lem:quasi} imply that a relative sub-gradient of any  ReLU regularizer can be evaluated with the chain rule. Since deep-learning frameworks such as PyTorch \cite{pytorch} and Tensorflow \cite{tensorflow2015} are built on formal application of the chain rule, calculating elements $G(\signal)$ with  $G(\signal) \in \partial_\phi \reg(\signal)$ can be done by using backpropagation. Thus, the backpropagation procedure is an appropriate choice for any form of gradient descent used to find critical points of the given functional satisfying~\eqref{eq:critpoint}. This is for example of interest for learned regularizers  \cite{obmann2021augmented, li2020nett, mukherjee2020learned, lunz2018adversarial}.

\begin{remark}[Learned regularizers as multi-well potentials] \label{re:multiwell}
    Using  (parametric) ReLU activation functions, the network~\eqref{eq:network1} is a composition of piecewise affine operators and as such itself a piecewise affine operator. This means that the regularizer  $\reg = \lVert \net(\cdot) \rVert^2 $ of  Theorem~\ref{thm:network}, as for example considered in \cite{li2020nett}, behaves like a ``multi-well potential'' similar to the one considered in Figure~\ref{fig:doublewell}. That is, it behaves like a function with multiple local minima where ideally each local minimum is located at a desirable solution.  
    
    A reasonable strategy to find such a regularizer is to train a network to have local minima which are located at the desired solutions. However, due to various difficulties during this process (e.g. the regularizer itself being only a local minimum of the loss function used for training, non-ideal network architectures) one would also expect the regularizer to have slightly different values at the desired solutions. This means, that even if the local minima are located at the desired solutions one cannot expect all of these local minima to have the same value much less expect each of these local minima to be global minima of the regularizer. To put this in different words, one should expect slight perturbations as in Figure~\ref{fig:doublewell}.
\end{remark}

\section{Numerical simulations} \label{sec:numerics}

The goal of this section is not to show that non-convex regularizers can improve the reconstructions, but rather to test the theory derived in the previous sections and to show what may happen when non-convex regularizers are chosen. 

To this end, we consider the discretzied version of two toy-problems in $1$D. We consider an inapinting problem where around $50 \%$ of the signal entries were randomly removed. In this case the kernel of the forward operator $\Ko$ is simple to compute and by using a separable prior we can easily study the properties of the solution we obtain in the limit. This makes the first toy problem ideal for testing whether the properties (as described in the theory section) of the limiting solution hold true or not.\\
Further, we consider recovering a signal from its cumulative sum. Since this matrix is invertible there is a unique solution and following Theorem~\ref{thm:convergence} we should observe convergence to this solution in the limit $\delta \to 0$. This toy problem is therefore well suited to study if the given $\phi$-critical  points actually converge to the unique solution.

For both problems we consider as the signal to recover the discretization of the function $f(t) = \exp(-t ^2) \cdot \cos(t) \cdot (t - 0.5) ^ 2 + \sin(t ^ 2)$ on $t \in [-1, 1]$ using $N = 512$ equidistant sample points. We denote this signal by $\signal_\mathrm{true}$ and the true underlying data by $\data_\mathrm{true} = \Ko \signal_\mathrm{true}$ where $\Ko$ is the forward operator of the corresponding problem.

For each problem we consider the similarity measure given by $\similarity(\signal, \data) = \frac{1}{2} \norm{\Ko \signal - \data}^2$ and we construct a regularizer by $\reg(\signal) = \sum_{i = 1}^N \psi_{\rho, \beta}(\signal_i)$. Here, we define $\psi_{\rho, \beta}(t) = (t - \rho)^2 \cdot (t + \frac{\rho}{2})^2 + \frac{\beta}{2} t^2$ for $\rho, \beta > 0$. The function $\psi_{\rho, \beta}$ is constructed in such a way that it is non-convex but relatively sub-differentiable, see Remark~\ref{re:examples}. Figure~\ref{fig:regularizer} shows the function $\psi_{\rho, \beta}(t)$ with parameters $\rho = 2$ and $\beta = 10^{-1}$ for $t \in [-3, 3]$ where the $y$-axis is plotted on a logarithmic scale in order to emphasize the non-convexity.
We can see that this function has a global minimum at around $t = -\frac{\rho}{2}$, a local minimum close to $t = \rho$ and another critical point in the interval $[0, 1]$. The parameters $\rho = 2$ and $\beta = 10^{-1}$ are used for all the following simulations.

As a separable sum of relatively sub-differentiable and non-convex terms the regularizer $\reg$ as defined above is relatively sub-differentiable and non-convex. By definition of $\reg$ it is further coercive and hence the functional $\tik_{\al, \data^\delta}$ is coercive. This shows that Condition~\ref{cond:a} is satisfied and we consider the stability and convergence of the $\phi$-critical  points according to Theorem~\ref{thm:stability} and \ref{thm:convergence}.

To simulate noisy data we consider the data $\data_k = \data_\mathrm{true} + \delta_k \cdot n$ where $n = \frac{\xi}{\norm{\xi}}$, $\xi$ is a normally distributed random variable and $\delta_k = 10^{-k}$ for $k \in \set{4, \dots, 14}$.

Since a bound $\phi$ for $\reg$ can be chosen such that $\reg'(\signal) \in \partial_\phi \reg(\signal)$ (see Remark~\ref{re:examples}), we can simply search for a classical critical point of $\tik_{\al, \data^\delta}$ in order to obatin $\phi$-critical points. To achieve this, we apply Newton's method, e.g. \cite{dennis1996numerical}, and we find an initial guess for Newton's method by applying Nesterov accelerated gradient descent \cite{nesterov1983method} to the starting point $\signal_0 = 0$.
In the following we denote by $\signal_\al^\delta$ a critical point of $\tik_{\al, \data^\delta}$ and by $\signal_\al$ a critical  point of $\tik_{\al, \data_\mathrm{true}}$. Here, $\signal_\al$ is considered as the limit point for the stability considerations for which we consider the choices $\al \in \set{10^{-2}, 10^{-3}, 10^{-4}}$. In order to test for convergence we chosen $\al = \al(\delta) = \delta^q$ for $q \in \set{1, \frac{3}{2}}$.
For the convergence simulations we consider as the limit point the signal $\signal_\mathrm{true}$ for the cumulative sum problem, as in this case the solution is unique, and we construct an approximate solution for the inpainting problem by finding a critical  point of the function $\tik_{\al(\delta), \delta}$ for $\delta = 10^{-16}$ and we denote this solution by $\signal_\plus$. Implementation details and code are publicly available\footnote{https://git.uibk.ac.at/c7021101/critical-point-regularization}.

\subsection{Results}

Figure~\ref{fig:cumsum} depicts the value $\norm{\signal_\al^\delta - \signal_\al}$ for different values of $\al > 0$ (left), $\norm{\Ko \signal_{\al(\delta), \delta} - \data_\mathrm{true}}$ (middle) and $\norm{\signal_{\al(\delta), \delta} - \signal_\plus}$ (right). Each of these values is plotted against $\delta$ on a log-log scale. The plot in the left shows that for any chosen $\al$ we can observe the convergence of the sequence $\signal_\al^\delta$ to the critical point $\signal_\al$ as the noise-level tends to $0$. The plots in the middle and right show the convergence behaviour for different choices $\al(\delta) = \delta^q$ as specified above. All of the sequences can be observed to converge, i.e. $\norm{\Ko \signal_{\al(\delta), \delta} - \data_\mathrm{true}} \to 0$ and $\norm{\signal_{\al(\delta), \delta} - \signal_\plus} \to 0$ as the noise $\delta$ tends to $0$.

\begin{figure}
    \centering
    \includegraphics[scale=0.3]{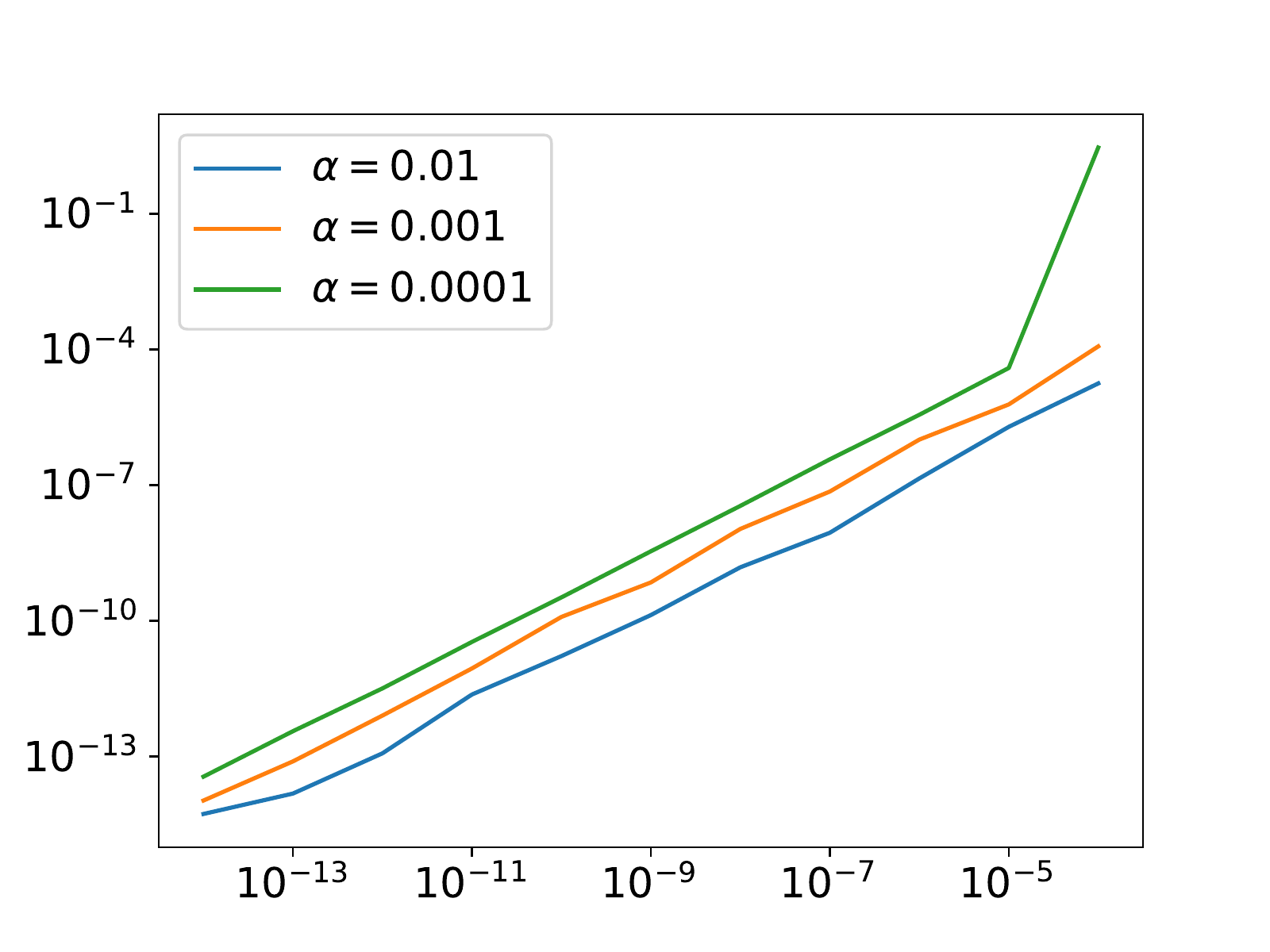}%
        \includegraphics[scale=0.3]{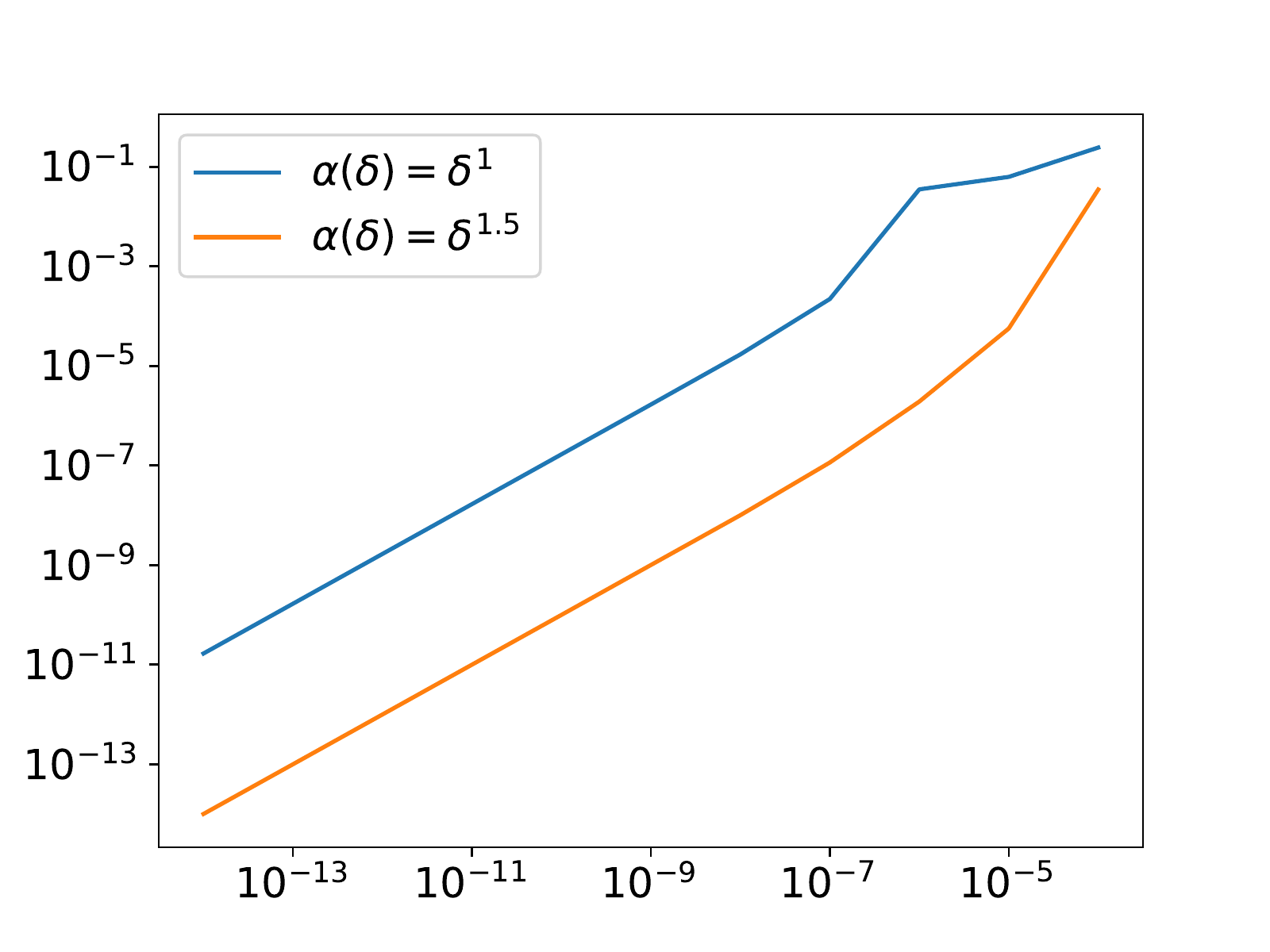}%
            \includegraphics[scale=0.3]{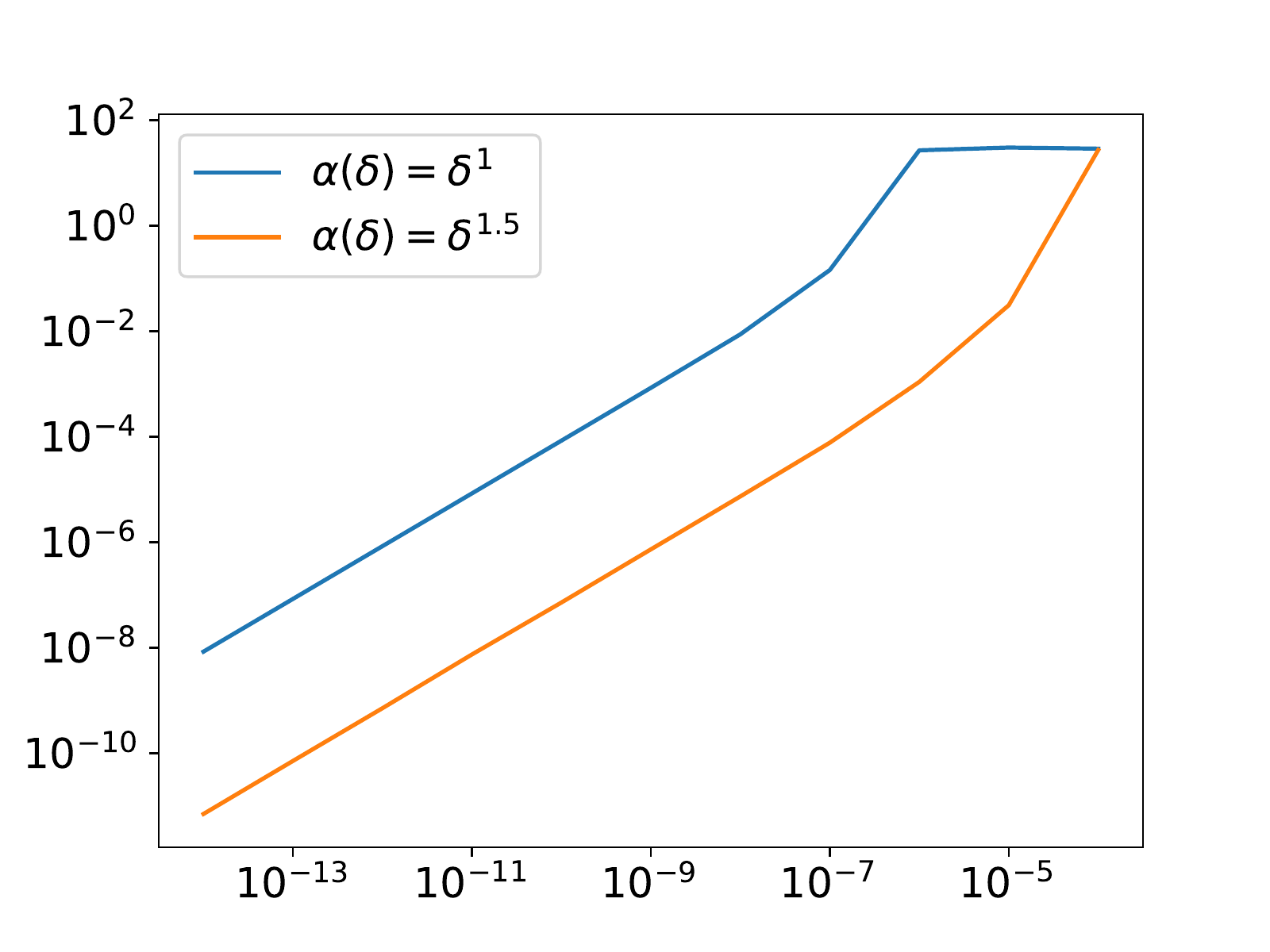}
    \caption{\textbf{Stability and convergence for the cumulative sum problem.} Each value is plotted in dependence on $\delta$. \textbf{Left:} $\norm{\signal_\al^\delta - \signal_\al}$ for different but fixed values of $\al$. \textbf{Middle:} $\norm{\Ko \signal_{\al(\delta), \delta} - \data}$ for different $\al(\delta)$. \textbf{Right:} $\norm{\signal_{\al(\delta), \delta} - \signal_\plus}$ for different $\al(\delta)$.}
    \label{fig:cumsum}
\end{figure}

Figure~\ref{fig:inpainting} shows the same behaviour for the stability and convergence plots for the inpainting problem as Figure~\ref{fig:cumsum} in the limit $\delta \to 0$. In particular convergence to a solution $\signal_\plus$ of the problem $\Ko \signal = \data_\mathrm{true}$ can be observed.

\begin{figure}
    \centering
    \includegraphics[scale=0.3]{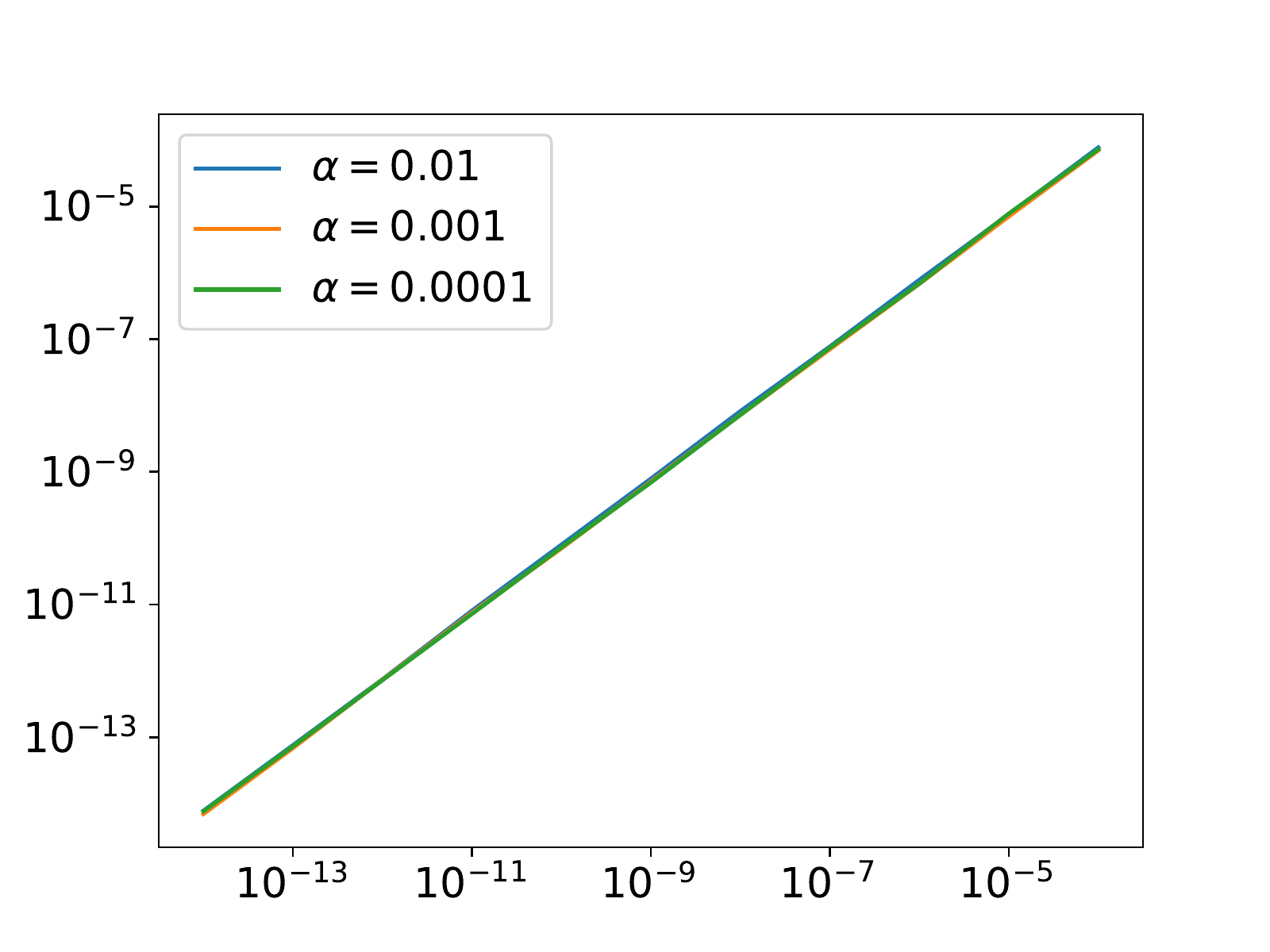}%
        \includegraphics[scale=0.3]{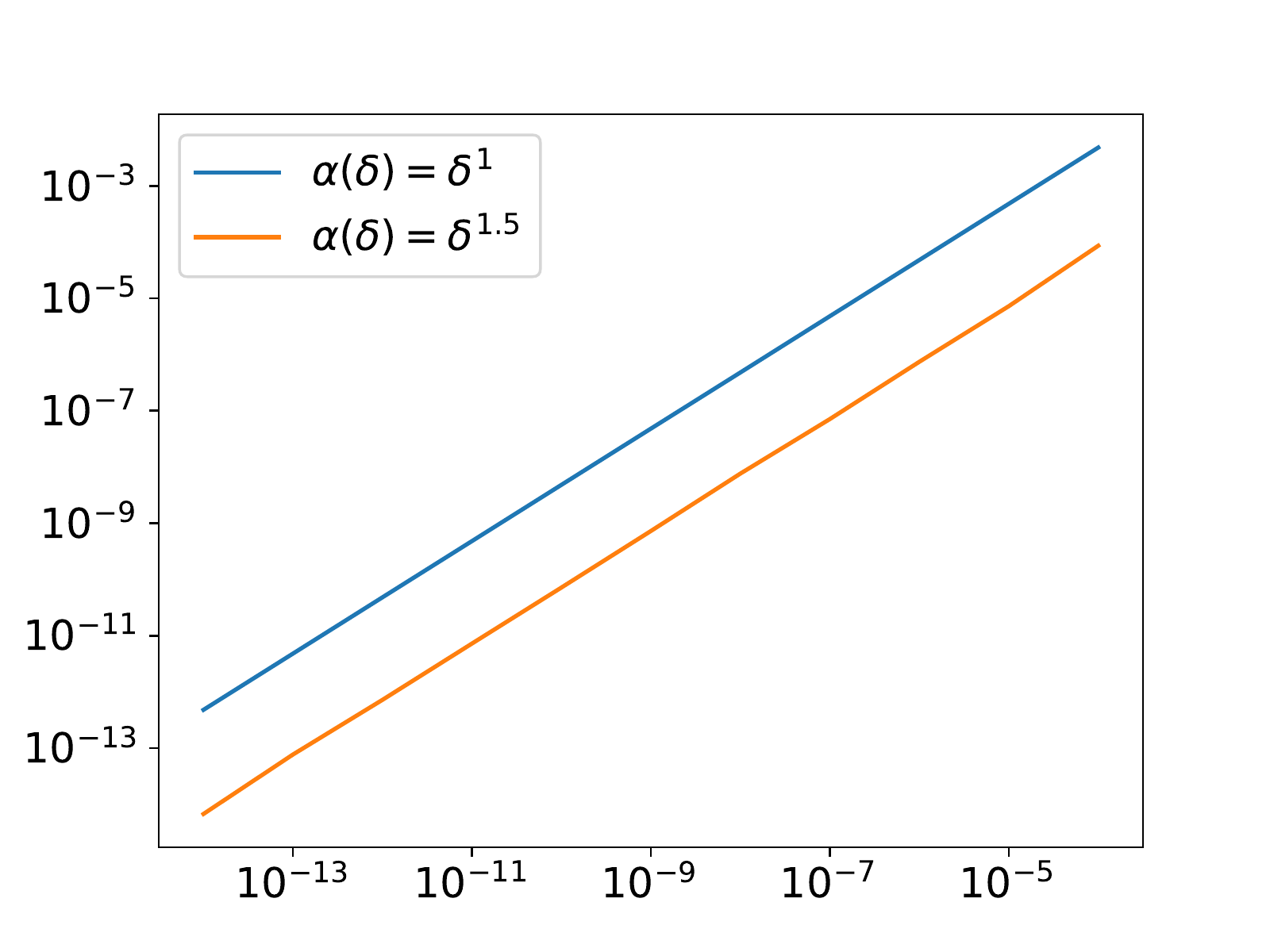}%
            \includegraphics[scale=0.3]{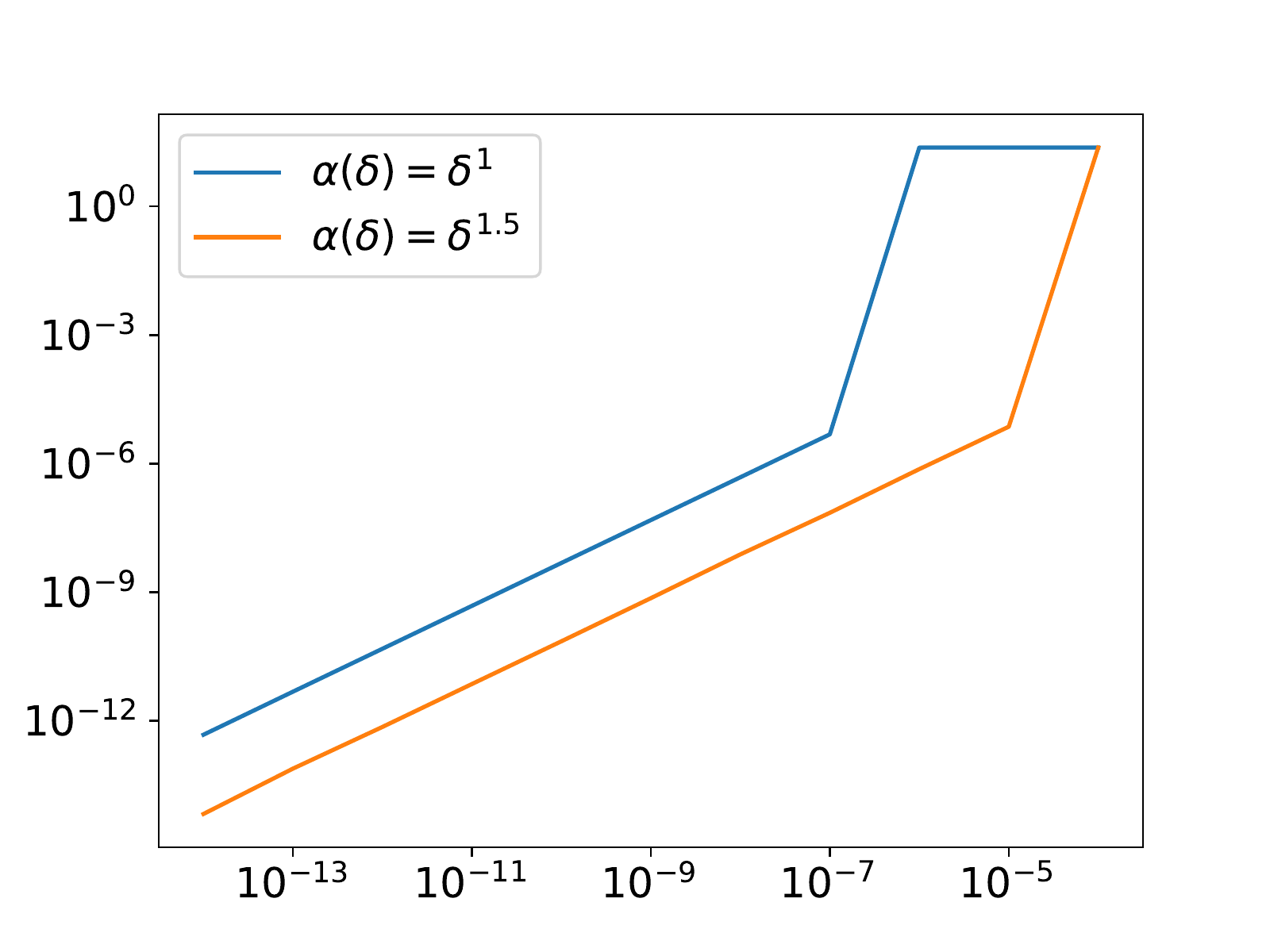}
    \caption{\textbf{Stability and convergence for the inpainting problem.} Each value is plotted in dependence on $\delta$. \textbf{Left:} $\norm{\signal_\al^\delta - \signal_\al}$ for different but fixed values of $\al$. \textbf{Middle:} $\norm{\Ko \signal_{\al(\delta), \delta} - \data}$ for different $\al(\delta)$. \textbf{Right:} $\norm{\signal_{\al(\delta), \delta} - \signal_\plus}$ for different $\al(\delta)$.}
    \label{fig:inpainting}
\end{figure}

A closer look at the inpainting problem reveals that the limit point $\signal_\plus$ is, however, not an $\reg$-minimizing solution. This can easily be checked due to the separability of the regularizer and the simple representation of the kernel of the inpainting problem.
The orange dot in Figure~\ref{fig:regularizer} (left) is the $\psi_{\rho, \beta}$-value of $(\signal_\plus)_i$ where $i$ is chosen as an index in the kernel of the inpainting matrix $\Ko$, i.e. such that $\Ko e_i = 0$ where $e_i$ is the $i$-th standard basis vector. Due to the separability of the regularizer we clearly have that $\signal_\plus$ is not an $\reg$-minimizing solution which arises due to the non-convexity of the regularizer. \\
Moreover, we have observed that if we initialize the values in the kernel close to $-1$ or $2$ then the limit $\signal_\plus$ will have entries at these $\phi$-critical  points of $\psi_{\rho, \beta}$. This shows that in such cases the solution we obtain in the limit heavily depends on the initialization we choose and that, depending on this initialization, the recovered solution may not be an $\reg$-minimizing solution and potentially even a local maximum or a saddle point.

Finally, Figure~\ref{fig:regularizer} (right) shows the values $\abs{\innerprod{\reg'(\signal_\plus), e_i}}$ where $(e_i)_i$ is a basis of the kernel of $\Ko$. Up to numerical accuracy we see that we have $\innerprod{\reg'(\signal_\plus), e_i} = 0$ for each such index $i$ which shows that $-\reg'(\signal_\plus) \in \ker(\Ko)^\perp$ as in Lemma~\ref{lemma:convergence}.

\begin{figure}
    \centering
    \includegraphics[scale=0.4]{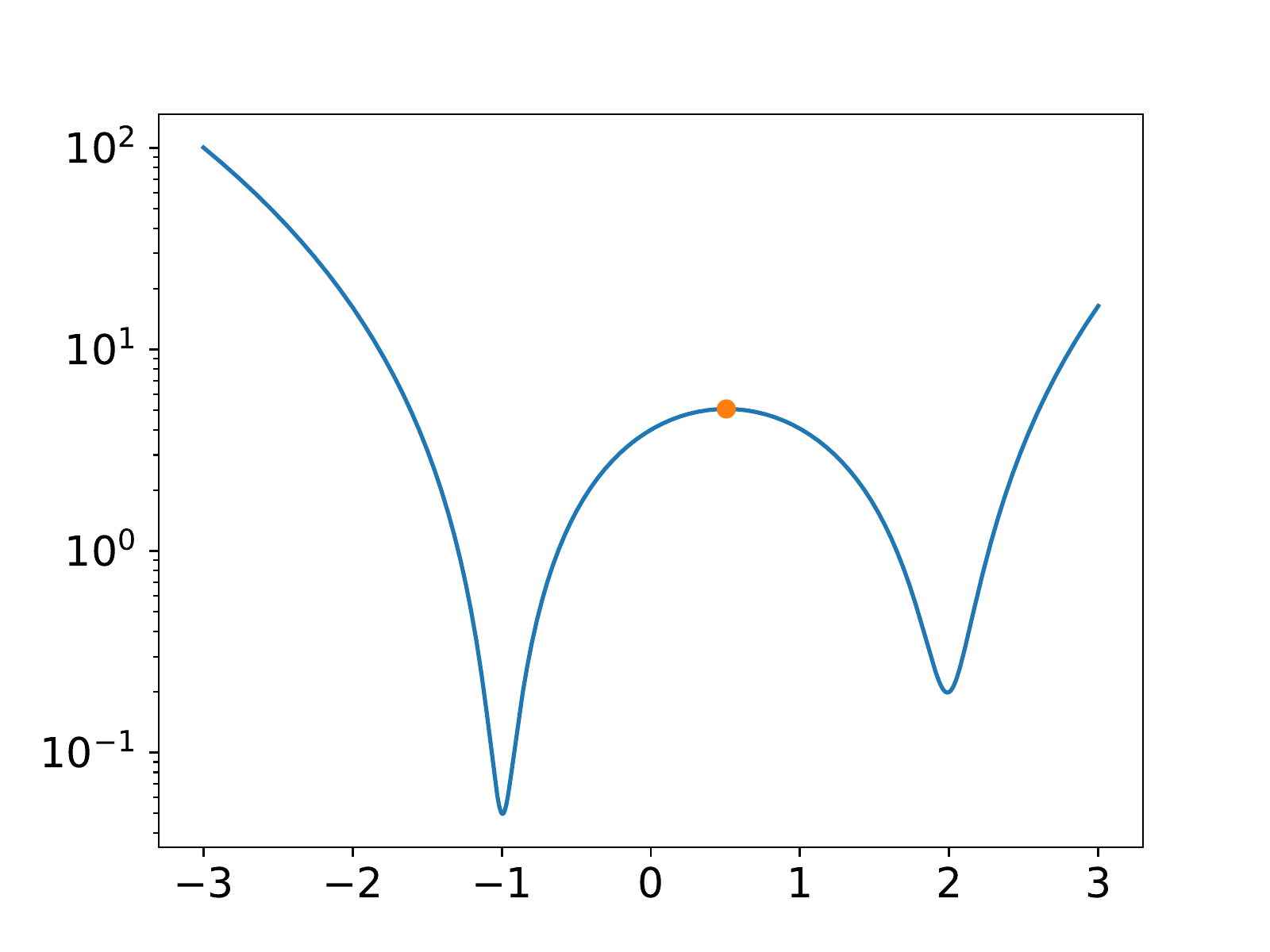}%
        \includegraphics[scale=0.4]{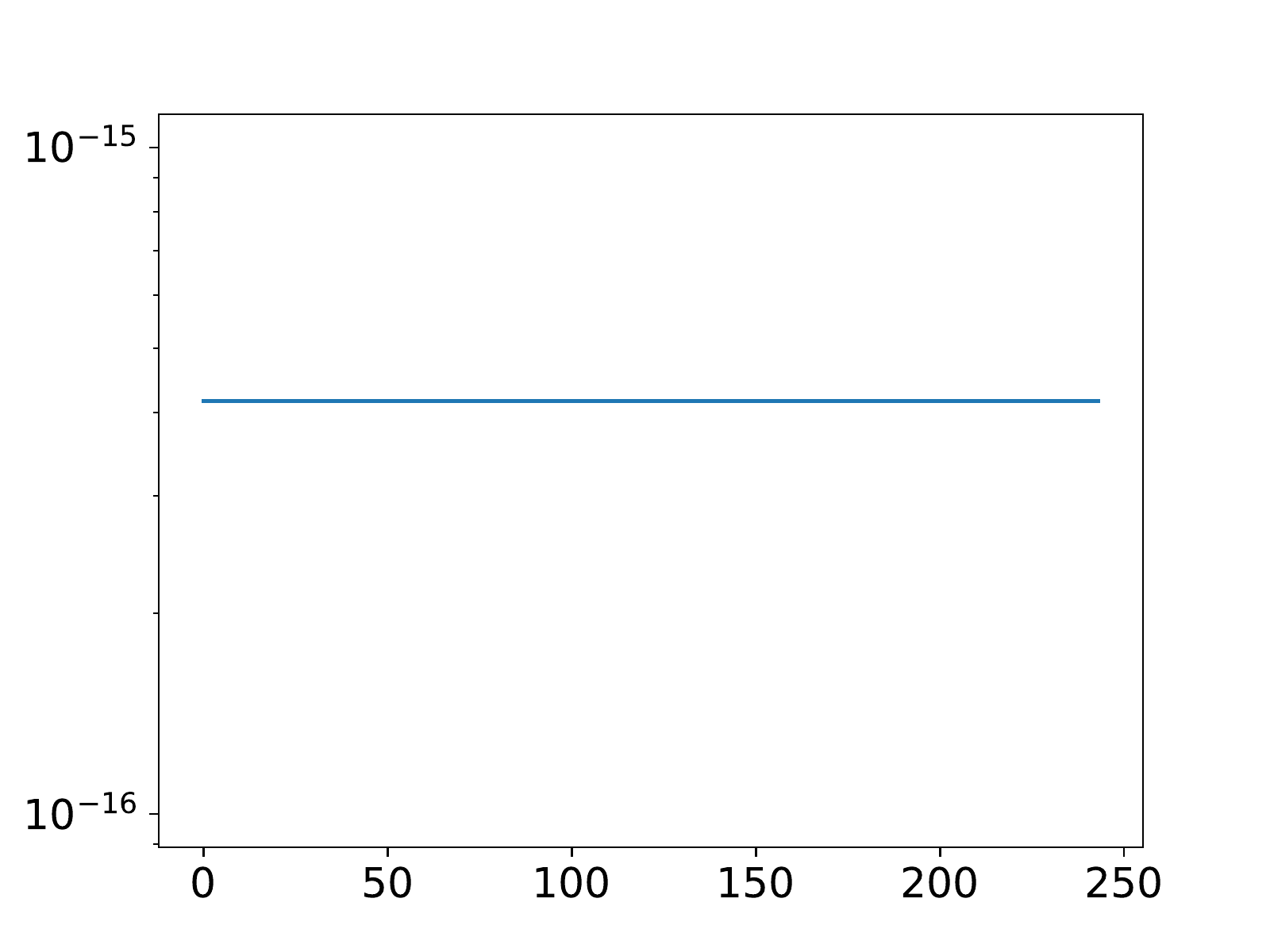}
    \caption{\textbf{Regularizer and properties of the inpainting solution.} \textbf{Left:} $\psi_{2, 10^{-1}}(t)$ for $t \in [-3, 3]$ on a logarithmic scale to emphasize the local minimum. The dot is the value of $(\signal_\plus)_i$ where $\signal_\plus$ is the solution of the inpainting problem and $i$ is an index chosen such that $e_i$ is in the kernel of the inpainiting problem. \textbf{Right:} the values $\abs{\innerprod{\reg'(\signal_\plus), e_i}}$ where $(e_i)_{i \in I}$ is a basis of the kernel of the inpainting problem.}
    \label{fig:regularizer}
\end{figure}

\section{Conclusion and outlook} \label{sec:conclusion}

We have introduced and studied the concept of critical point regularization, which, opposed to classical variational regularization, considers ($\phi$-)critical points of Tikhonov-functionals as regularized solutions. The advantage of this approach is that it completely discards the strong and typically unrealistic assumption of being able to achieve global minimizers of these functionals. Our theory shows that under reasonable assumptions on the involved functionals the resulting method will nevertheless be a stable and convergent regularization method. Further, we have shown that the solutions in the limit $\delta \to 0$ satisfy some form of first order optimality conditions of the constrained optimization problem $\inf_\signal \reg(\signal)$ subject to the constraint $\similarity(\signal, \data) = 0$. Besides this, the theory presented here extends the theory of convex functionals by showing that at no point does one require global minimizers, but only points which are close to a global minimum in some sense. For practical applications this means that minimization algorithms do not need to be run until convergence but may be stopped early, if easily verifiable conditions are met.
Additionally, under assumptions on the regularizer $\reg$ this theory is directly applicable to regularized solutions which are classical critical points of the involved functionals. As such our theory gives stability and convergence results for critical points of potentially non-convex functionals.

Finally, we have provided numerical simulations which support our theoretical findings, i.e. the stability and convergence of critical point regularization. Depending on the algorithm used for obtaining critical points, these numerical examples show that one cannot expect to find global or even local minima which further supports the arguments for the need of a theory based on ($\phi$-)critical points, which we have developed in this paper.

As the main concern of this paper was to introduce the concept of using ($\phi$-)critical points as regularized solutions, we have not derived any stability- or convergence-rates and studying such rates is subject to future work. Besides this, deriving conditions under which learned regularizers, e.g. \cite{antholzer2019nett, obmann2021augmented, mukherjee2020learned}, give rise to relatively sub-differentiable functions is also subject of future work.

\bibliographystyle{abbrv}
\bibliography{refs}

\end{document}